\documentclass[10pt]{amsart}
\usepackage{main}

\def\grad{\mathop{\rm grad}}
\def\bar{\overline}
\def\hat{\widehat}
\def\tilde{\widetilde}
\def\D{\mathcal D}

\title[Inverse obstacle problem for the wave equation]{Inverse obstacle problem for the non-stationary wave equation with an unknown background}
\author{Lauri Oksanen}
\address{University of Helsinki, P.O. Box 68 FI-00014}
\email{lauri.oksanen@helsinki.fi}
\date{\today}
\begin{document}
\maketitle
\begin{abstract}
We consider boundary measurements for the wave equation on a bounded domain $M \subset \R^2$ or on a compact Riemannian surface, and introduce a method to locate 
a discontinuity in the wave speed.
Assuming that the wave speed consist of an inclusion in a  known smooth background, the method can determine the distance from any boundary point to the inclusion. In the case of a known constant background wave speed, the method reconstructs a set contained in the convex hull of the inclusion and containing the inclusion. 
Even if the background wave speed is unknown, the method can reconstruct the distance from each boundary point to the inclusion assuming that the Riemannian metric tensor determined by the wave speed gives simple geometry in $M$.
The method is based on reconstruction of volumes of domains of influence by solving a sequence of linear equations. 
For $\tau \in C(\p M)$ the domain of influence $M(\tau)$ is 
the set of those points on the manifold from which 
the distance to some boundary point 
$x$ is less than $\tau(x)$.
\end{abstract}

\section{Introduction and the statement of the results}

Let us consider the wave equation on a compact set $M \subset \R^2$,
\begin{align}
\label{eq:wave_isotropic}
&\p_t^2 u(t,x) - \tilde c(x)^2 \Delta u(t,x) = 0, 
\quad &(t,x) \in (0, \infty) \times M,
\\&u(0,x) = 0,\ \p_t u(0,x) = 0,
\quad &x \in M, \nonumber
\\&\p_\nu u(t,x) = f(t,x), 
\quad &(t,x) \in (0, \infty) \times \p M, \nonumber
\end{align}
with a piecewise smooth wave speed 
\begin{equation*}
\label{piecewise_wave_speed}
\tilde c(x) = 
\begin{cases}
c_0(x) + c_\Sigma(x), & x \in \Sigma,
\\ c_0(x), & x \in M \setminus \Sigma,
\end{cases}
\end{equation*}
where $\Sigma \subset M^{int}$ and
$c_0$ and $c_\Sigma$ are smooth 
strictly positive functions.
Here $\p_\nu$ denotes the normal derivative 
and the boundaries $\p M$ and 
$\p \Sigma$ are smooth.
Physically $\Sigma$ corresponds to an obstacle or an inclusion in which acoustic waves propagate faster than in the background medium modelled by $c_0$.

Let $\Sigma$ and $c_\Sigma$ be unknown
and let us assume either that the wave speed $c_0$ is known or that it is unknown and gives simple\footnote{We recall the definition of simple geometry below, see Definition \ref{def:simple}.} geometry in $M$. 
We describe a method to locate the inclusion $\Sigma$
using the operator
\begin{equation}
\label{dtn}
\Lambda_{2T} : f \mapsto u^f|_{(0,2T) \times \p M},
\quad f \in L^2((0,2T) \times \p M),
\end{equation}
where $u^f(t,x) = u(t,x)$ is the solution of (\ref{eq:wave_isotropic}) and $T > 0$ is large enough.
The operator $\Lambda_{2T}$ models
boundary measurements and is called the 
Neumann-to-Dirichlet operator.
The method works equally well with 
an anisotropic background wave speed, 
see the wave equation (\ref{eq:wave}) below.

In a recent article \cite{Chen2010}, 
Chen, Haddar, Lechleiter and Monk introduce a method to reconstruct an impenetrable obstacle in the Euclidean background by using time domain acoustic measuments. 
Their method is similar to the linear sampling method originally developed for an inverse obstacle scattering problem in the frequency domain \cite{Colton1996}.
The method we propose here is based solely on control theoretic approach in the time domain.
Our method is similar to the iterative time-reversal control method by Bingham, Kurylev, Lassas and Siltanen \cite{Bingham2008, Dahl2009}
originally developed to reconstruct a smooth wave speed as a function. 
Computationally our method consists of 
solving a sequence of Tikhonov regularized linear equations on $L^2((0, 2T) \times \p M)$, 
and allows for a very efficient implementation if computation steps are intertwined with measurement steps. 

By using the boundary control method,
a smooth wave speed can be fully reconstructed from the Neumann-to-Dirichlet operator.
This is a result by Belishev \cite{Belishev1987}
for an isotropic wave speed and 
by Belishev and Kurylev \cite{Belishev1992} for an anisotropic wave speed. 
However, the boundary control method in its original form  is exponentially unstable and hard to regularize.
To our knowledge, there are two computational implementations of the method \cite{Belishev1999, Kabanikhin2005}.
In addition to these two implemetations, 
the only numerical results related to the boundary control method we are aware of are in the recent article by Pestov, Bolgova and Kazarina \cite{Pestov2010}.
We believe that there is a demand for methods that reconstruct less but are more robust. 

Our method locates the inclusion $\Sigma$ by computing the 
travel time distance from each boundary point to $\Sigma$.
Let us next describe the travel time distance function in detail. 
We denote by $(\hat M, g)$ a smooth complete and connected Riemannian surface that models the background wave speed. 
For example, in the isotropic case (\ref{eq:wave_isotropic})
we have $\hat M = \R^2$ and 
\begin{align*}
g(x) := \frac{(dx^1)^2 + (dx^2)^2}{c_0(x)^2},
\quad x \in \R^2.
\end{align*}
We let $M \subset \hat M$ and $\Sigma \subset M^{int}$
be compact sets with smooth boundary and non-empty interior.
Moreover, we assume that $M$ is connected
and that the wave speed with the inclusion is given by the non-smooth Riemannian metric tensor, 
\begin{equation}
\label{g_tilde}
\tilde g_{jk}(x) = 
\begin{cases}
c(x)^{-2} g_{jk}(x), & x \in \Sigma,
\\ g_{jk}(x), & x \in M \setminus \Sigma,
\end{cases}
\quad j,k =1, 2,
\end{equation}
where $c$ is a smooth scalar function on $\Sigma$ satisfying $c(x) > 1$ for all $x \in \Sigma$.
We denote by $d(x, y)$, $x, y \in M$, 
the Riemannian distance function of $(M, g)$ and by $\hat d$
and $\tilde d$ the Riemannian distance functions of $(\hat M, g)$ and of $(M, \tilde g)$, respectively.
Note that $d$ is not necessarily a restriction of $\hat d$.
For example, if $g$ is the Euclidean metric tensor and $M$ is non-convex then there exist $x, y \in \p M$ such that $\hat d(x, y) < d(x, y)$.

In the case of a known background manifold $(\hat M, g)$ without boundary, our method reconstructs the {\em boundary distance hull} of $\Sigma$,
\begin{align*}
H_{\p M}(\Sigma) 
:= 
M \setminus \bigcup_{y \in \p M} 
\hat B(y, \hat d(y, \Sigma)),
\end{align*}
where $\hat B(y, r) := \{ x \in \hat M;\ \hat d(x, y) < r\}$ for $y \in \hat M$ and $r > 0$.
If $(M, g)$ is simple, then the embedding $M \subset \hat M$ plays no role in our method, and we can replace $\hat d$ with $d$ in the definition of $H_{\p M}(\Sigma)$.
Note that $\Sigma \subset H_{\p M}(\Sigma)$ and, in the special case of the Euclidean background, the boundary distance hull $H_{\p M}(\Sigma)$ is a subset of the convex hull of $\Sigma$.

In the case of an unknown simple background manifold 
$(M, g)$ with known boundary $(\p M, g|_{\p M})$, our method reconstructs the distance function
\begin{align*}
r_\Sigma(y) := d(y, \Sigma), 
\quad 
y \in \p M.
\end{align*}
Thus it determines the set
\begin{align*}
\mathcal B_{\p M}(\Sigma) := 
\{ (y, \eta) \in \p TM;\ |\eta|_g < d(y, \Sigma)\},
\end{align*}
where $\p TM := 
\{(y, \eta) \in TM;\ y \in \p M,\ \eta \in T_y M \}$.
The set $\mathcal B_{\p M}(\Sigma)$ satisfies
\begin{align}
\label{hull_in_pTM}
H_{\p M}(\Sigma) 
= M \setminus \exp(\mathcal B_{\p M}(\Sigma)),
\end{align}
where $\exp$ is the exponential map of $(M, g)$.
%
As $g$ is unknown, the equation (\ref{hull_in_pTM})
does not give a method to determine $H_{\p M}(\Sigma)$.
However, if we know a priori that $g$ is close to a known metric tensor $g^0$ on $M$, 
we get a distorted image of $H_{\p M}(\Sigma)$ 
simply by visualizing the set
$\exp^0 \mathcal B_{\p M}(\Sigma)$, where
$\exp^0$ is the exponential map of $(M, g^0)$.

\begin{figure}[t]
\centering
\includegraphics[width=0.4\textwidth,clip]{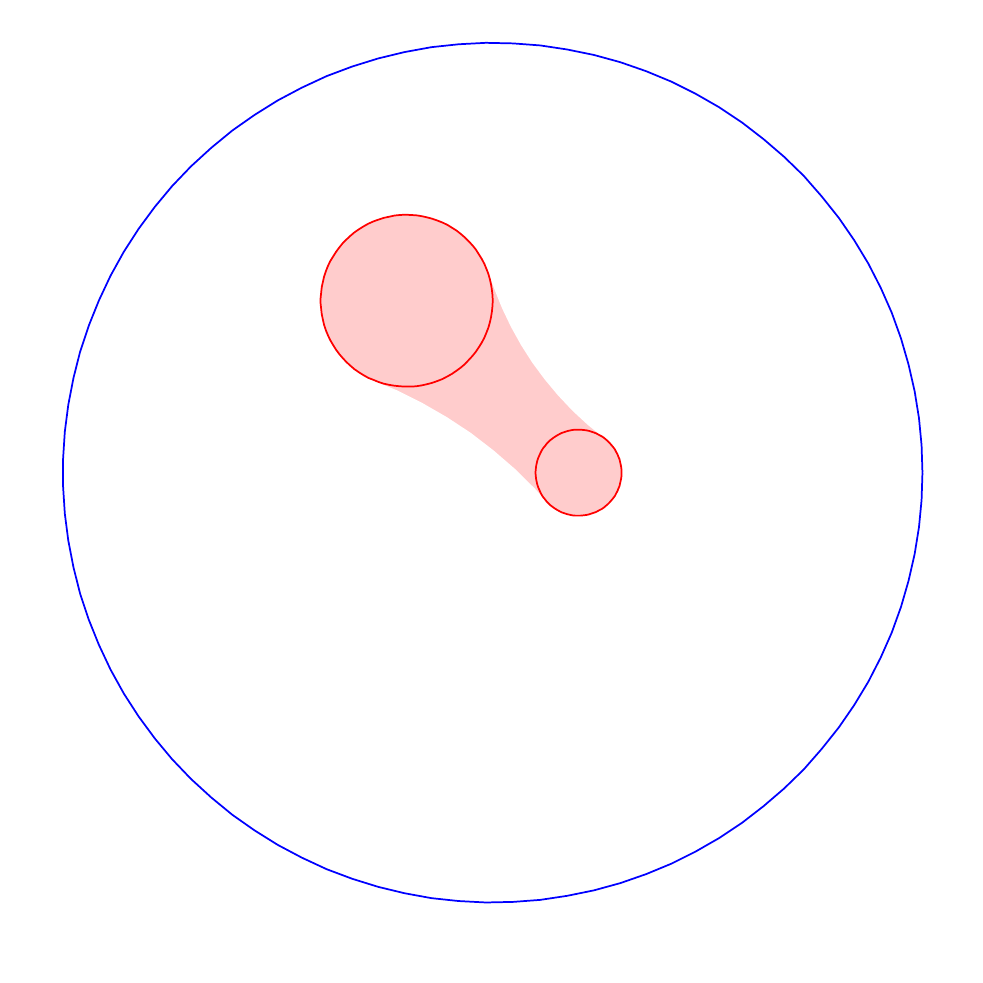}
\includegraphics[width=0.4\textwidth,clip]{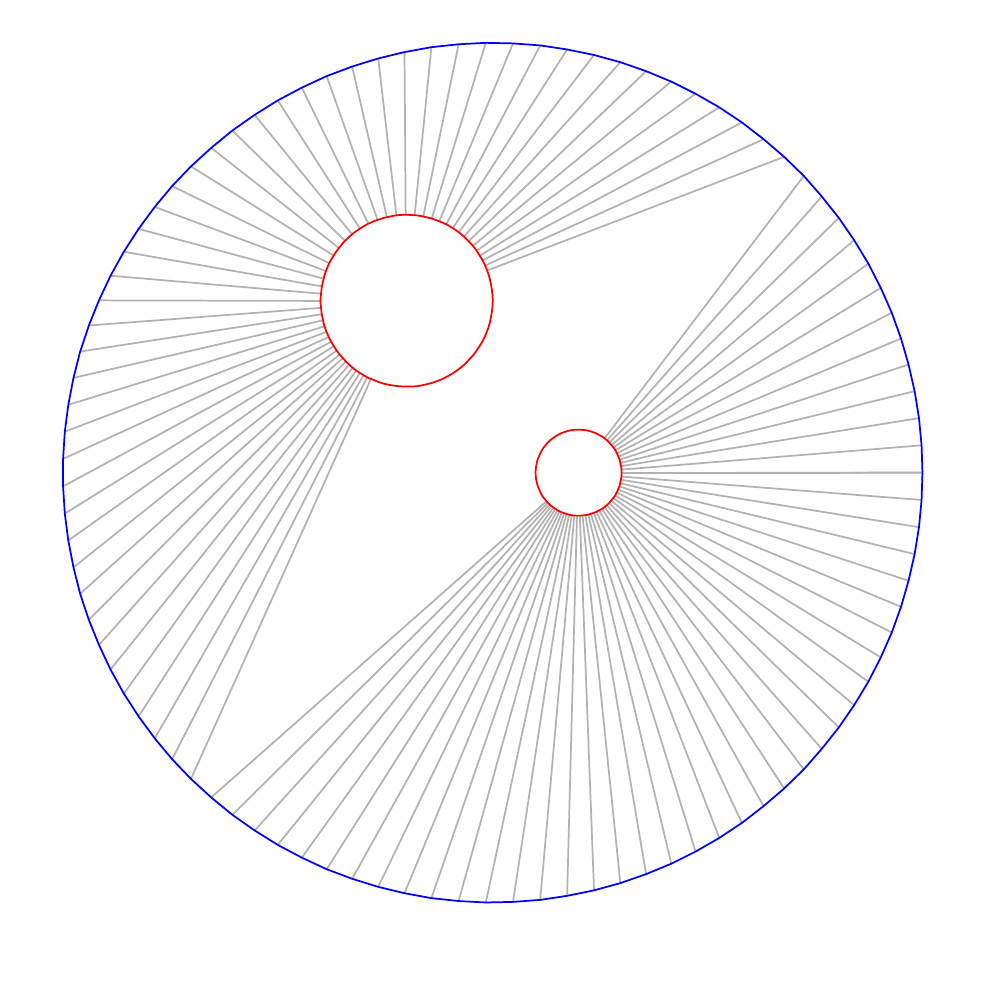}
\caption{
The domain $M$ is the large disc and 
the background wave speed is constant. The inclusion $\Sigma$ consist of the two small discs. 
{\em On left.} If the backgroud is known, the method can reconstruct the boundary distance hull $H_{\p M}(\Sigma)$ visualized as the shaded area in the picture. 
{\em On right.} If the background is unknown and simple, the method can reconstruct the vectors $\hat r_\Sigma(y) \grad \hat r_\Sigma(y)$, $y \in \p M$, 
visualized as the line segments in the picture.
}
\label{fig_hull}
\end{figure}

Moreover, if $\hat M$ is also simple and $g|_{\hat M \setminus M}$ is known, then the method can reconstruct 
the distance function,
\begin{align*}
\hat r_\Sigma(x) := \hat d(x, \Sigma), 
\quad 
x \in \hat M \setminus M.
\end{align*}
From this function it is easy to extract also directional information, since $\grad \hat r_\Sigma(x)$ exists for almost all $x \in \hat M \setminus M$ and is a unit vector pointing to such a point $z \in \p \Sigma$ that 
\begin{align*}
\hat d(x, z) = \hat d(x, \Sigma).
\end{align*}
See Figure \ref{fig_hull} for a visualization of 
$H_{\p M}(\Sigma)$ and the directions $\grad \hat r_\Sigma(y)$, $y \in \p M$.
In this paper, we focus on the recovery of the distance function $\hat r_\Sigma$ and do not study further the directional information contained in $\hat r_\Sigma$.

It is well-known that the high frequency behavior of the scattering pattern determines a convex obstacle in the Euclidean background.
For a review of this and related results we refer to the survey article \cite{Petkov2003}. 
We emphasize, however, that our method does not rely on analysis of high-frequency solutions.
Moreover, it seems possible that a logarithmic type stability result for the proposed method could be proven by using a stability estimate for the hyperbolic unique continuation principle. Such an estimate is formulated in an unfinished manuscript by Tataru \cite[Thm. 3.45]{Tataru}. 
Furthermore, our method is similar to the iterative time-reversal control method, and that method can be modified to work in the presence of measurement noise \cite{Bingham2008}. Such robustness against noise is typically not possible for a high-frequency solutions based method.

From a point of view of numerical computations, our method consists of reconstruction of volumes of domains of influence. We define for a function $\tau : \p M \to [0, \infty)$ the {\em domain of influence}
with and without the inclusion,
\begin{align*}
\tilde M(\tau) &:= \{x \in M;\ \text{there is $y \in \p M$ such that $\tilde d(x,y) \le \tau(y)$}\},
\\
M(\tau) &:= \{x \in M;\ \text{there is $y \in \p M$ such that $d(x,y) \le \tau(y)$}\}.
\end{align*}
Moreover, we define
\begin{equation*}
C_T(\p M)
:= 
\{ \tau \in C(\p M);\ 
\text{$0 \le \tau(x) \le T$ for all $x \in \p M$} \},
\end{equation*}
and denote by $\tilde m$ and $m$ the Riemannian volume measures of $(M, \tilde g)$ and $(M, g)$, respectively. Using the method introduced in \cite{Oksanen2011}
we can compute the volumes, 
\begin{align}
\label{volume_data}
\tilde m(\tilde M(\tau)), \quad \tau \in C_T(\p M),
\end{align}
from the operator $\Lambda_{2T}$ by solving a sequence of linear equations on the space $L^2((0, 2T) \times \p M)$.
We outline this method in section \ref{sec_minimization}, where
we also generalize it to cover a wave speed given by a non-smooth Riemannian metric tensor of the form (\ref{g_tilde}).

Note that, on the one hand 
$\tilde M(\tau)$ grows as 
$c(x)$ in (\ref{g_tilde}) grows,
but on the other hand the factor $c(x)^{-2}$
in the volume measure $\tilde m(x) = c(x)^{-2} m(x)$
gets smaller.
We show that the latter effect is dominating
near the boundary $\p \Sigma$.
That is, we show the following result in section 
\ref{sec_inclusion_detection}.

\begin{theorem}
\label{thm_main}
Let $\tau : \p M \to \R$ be such that 
\begin{equation}
\label{boundary_intersection}
M(\tau) \cap \Sigma \ne \emptyset,
\quad 
M(\tau) \cap \Sigma^{int} = \emptyset.
\end{equation}
Suppose that there is a 
neighborhood $U_\Sigma$ of $\Sigma$
such that 
$M(\tau) \cap U_\Sigma$ is
an embedded smooth manifold with boundary.
Then 
\begin{equation}
\label{liminf_test}
\liminf_{\epsilon \to 0}
\frac{m(M(\tau + \epsilon)) - \tilde m(\tilde M(\tau + \epsilon))}{\epsilon^{3/2}}
> 0. 
\end{equation}
\end{theorem}

If the background manifold $(M, g)$ is known, $y_0 \in \p M$ and the sets, 
\begin{equation*}
B(y_0, r) := \{x \in M^{int};\ d(x, y_0) < r \}, 
\quad r > 0,
\end{equation*}
have smooth boundaries, then 
Theorem \ref{thm_main} gives a test to determine the smallest $r > 0$ such that 
$B(y_0, r) \cap \Sigma \ne \emptyset$, that is,
the distance $d(y_0, \Sigma)$.
Indeed, we may choose $\tau_r \in C(\p M)$
such that $\tau_r(y_0) = r$ and
$\tau_r = 0$ outside a small neighborhood of $y_0$
in $\p M$.
If $\tau_r(y)$ decreases to zero fast enough as $d(y, y_0)$
grows, then $M(\tau_r) = B(y_0, r)$.
Moreover, $M(\tau_r) = \tilde M(\tau_r)$
whenever $M(\tau_r) \cap \Sigma = \emptyset$,
whence by Theorem \ref{thm_main},
\begin{equation*}
d(y_0, \Sigma) = \max\{r > 0;\ m(M(\tau_r)) = \tilde m(\tilde M(\tau_r)) \}.
\end{equation*}

In section \ref{sec_distance} we present a refinement of this test that works for compact domains $M \subset \hat M$ with smooth boundary, where $\hat M$ is a complete smooth manifold without boundary. In particular, $(M, g)$ may have conjugate points and $B(y_0, r)$ may have non-smooth boundary for some $y_0 \in \p M$ and $r > 0$.
Moreover, we show that in the case of an unknown and simple background manifold $(M, g)$, the distance $d(y_0, \Sigma)$ can be reconstructed by using a test related to the second derivative of the function $r \mapsto \tilde m(B(y_0, r))$.

Let us next summarize our results as a theorem.
We remind the reader that $(\hat M, g)$ is 
a smooth complete and connected Riemannian surface,
$M \subset \hat M$ and $\Sigma \subset M^{int}$
are compact sets with smooth boundaries and non-empty interiors,
$M$ is connected and $\tilde g$ is defined 
by (\ref{g_tilde}).
We consider the wave equation 
\begin{align}\label{eq:wave}
&\p_t^2 u(t,x) - \Delta_{\tilde g} u(t, x) = 0, \quad (t,x) \in (0,\infty) \times M,
\\\nonumber& u|_{t=0} = 0, \quad \p_t u|_{t=0}=0,  
\\\nonumber& \p_\nu u(t,x) = f(t,x), \quad (t,x) \in (0,\infty) \times \p M,
\end{align}
where $\Delta_{\tilde g}$ is the Laplace-Beltrami operator on $(M, \tilde g)$ and $\p_\nu$ is the normal derivative on $(\p M, g)$. 
We denote by $|\tilde g(x)|$ the determinant of 
$\tilde g(x)$. If 
$(\tilde g^{jk}(x))_{j,k=1}^2$ denotes the inverse of $\tilde g(x) = (\tilde g_{jk}(x))_{j,k=1}^2$ in some coordinates, then
\begin{align*}
\Delta_{\tilde g} u
&:= \sum_{j,k=1}^2 |\tilde g(x)|^{-\frac 12}\frac {\p}{\p x^j} 
\ll( |\tilde g(x)|^{\frac 12}\tilde g^{jk}(x)\frac {\p u}{\p x^k} \rr),
\\\p_\nu u 
&:= \sum_{j,k=1}^2 g^{jk}(x) \nu_k(x) \frac{\p u}{\p x^j},
\end{align*}
where $\nu = (\nu_1, \nu_2)$ is the exterior co-normal vector of $\p M$
normalized with respect to $g$, that is, $\sum_{j,k=1}^2 g^{jk}\nu_j\nu_k=1$.
We define the operator $\Lambda_{2T}$
by (\ref{dtn}) where $u^f$ is the solution of
(\ref{eq:wave}).

\begin{definition}
\label{def:simple}
A compact Riemannian manifold $(M, g)$ with boundary is {\em simple}
if it is simply connected, any geodesic has no conjugate points and
$\p M$ is strictly convex with respect to the metric $g$.
\end{definition}

\begin{theorem}
\label{thm_summary}
If the boundary $(\p M, g|_{\p M})$ 
and the operator $\Lambda_{2T}$ 
are known, then the volume data (\ref{volume_data})
can be computed.
Moreover, if $T > \norm{r_\Sigma}_{L^\infty(\p M)}$,
then the following two implications hold.
\begin{itemize}
\item[(i)] If $(M, g)$ is simple, 
then the set $\mathcal B_{\p M}(\Sigma) \subset \p TM$
can be reconstructed from the volume data (\ref{volume_data}). Furthermore, the boundary distance hull satisfies
\begin{align*}
H_{\p M}(\Sigma) 
= 
M \setminus \exp (\mathcal B_{\p M}(\Sigma)).
\end{align*}
\item[(ii)] If $(\hat M, g)$ has no boundary and $g$ is known,
then the boundary distance hull 
$H_{\p M}(\Sigma)$
can be reconstructed from the volume data (\ref{volume_data}).
\end{itemize}
\end{theorem}

In the context of inverse acoustic scattering problems, there is an extensive literature about inclusion detection methods. Some of these methods have also been applied to 
solve inverse obstacle problems for the time domain wave equation.
The methods in \cite{Lines2005} and in \cite{Luke2006} take Fourier transforms of time domain measurement data and solve the inverse obstacle problem by using inclusion detection methods developed for scattering problems. Moreover, the method in \cite{Burkard2009} process the measurement data partly in the frequency domain.

The only inclusion detection method processing the measurement data entirely in the time domain that we are aware of is the already mentioned sampling method in \cite{Chen2010}.
The analysis in \cite{Chen2010} depends on frequency domain techniques, and the finite speed of propagation for the wave equation seems to be an obstruction in carrying out the analysis.
On the contrary, our method is based on the finite speed of propagation and the complementary unique continuation principle by Tataru \cite{Tataru1995}.

Well-known inclusion detection methods in the frequency domain include the already mentioned 
linear sampling method by Colton and Kirsch and the 
enclosure method by Ikehata \cite{Ikehata2000}.
A modification of the linear sampling method 
by Kirsch is called the factorization method \cite{Kirsch1998}, and it can be interpreted by using localized potentials \cite{Gebauer2008}.
The enclosure method is the first inclusion detection method based on the complex geometrical optics solutions developed by Sylvester and Uhlmann in their fundamental paper \cite{Sylvester1987}. For later complex geometrical optics solutions based methods see \cite{Ikehata2004a, Ide2007, Uhlmann2008}.

The factorization method has been applied also to electrostatic measurements \cite{Hahner1999, Bruhl2001},
and the enclosure method was developed for both acoustic scattering and electrostatic measurements from the very beginning.
For other methods to solve inverse obstacle problems 
related to scattering and electrostatic measurements see the probe \cite{Ikehata1998}
and singular sources \cite{Potthast2001} methods,
the no response test \cite{Luke2003},
the scattering support techniques \cite{Kusiak2003,  Potthast2003, Hanke2008} and the review article \cite{Potthast2006}.
Furthermore, we refer to the review article \cite{Isakov2009} for uniqueness and stability results related to inverse obstacle problems.

The uniqueness results for the inverse problem for the wave equation mentioned above assume smooth wave speed \cite{Belishev1987, Belishev1992}. 
However, in a recent article \cite{Kirpichnikova2007}, Kirpichnikova and Kurylev consider piecewise smooth wave speeds on Riemannian polyhedra. 
Moreover, the stability results \cite{Anderson2004, Bellassoued2010, Stefanov2005} establish uniqueness for wave speeds with a limited number of derivatives,
and there is an extensive literature about uniqueness results for the related Calder\'on's inverse problem under low regularity assumptions including
\cite{Astala2006, Brown2003, Brown1997, Greenleaf2003, Kohn1985, Paivarinta2003}.
For a review of the latter results we refer to \cite{Uhlmann2009}.

\section{Inclusion detection from the volume data}
\label{sec_inclusion_detection}

In this section we prove Theorem \ref{thm_main}.
We assume throughout the section 
that $\tau : \p M \to \R$ satisfies (\ref{boundary_intersection}).

%

\begin{lemma}
\label{lem_interesting_x}
Let $\epsilon > 0$. 
If $x \in \tilde M(\tau + \epsilon)$ and
\begin{equation}
\label{interesting_x}
x \in \Sigma 
\quad \text{or} \quad 
x \notin M(\tau + \epsilon),
\end{equation}
then there is 
$z \in \p \Sigma$ 
such that
\begin{equation}
\label{intersection_point}
\tilde d(x, z) + d(z, M(\tau)) \le \epsilon.
\end{equation}
\end{lemma}
\begin{proof}
Let $l(\gamma)$ denote the length of a path $\gamma$
with respect to metric $d$ 
and $\tilde l(\gamma)$ with respect to metric $\tilde d$.
There is $y \in \p M$ and a path 
$\gamma : [0, \ell] \to M$
from $y$ to $x$ such that $\tilde l(\gamma) \le \tau(y) + \epsilon$.

We claim that both of the conditions (\ref{interesting_x})
imply that $\gamma$ intersects $\Sigma$.
First, if $x \in \Sigma$ then
this is immediate.
Second, if $x \notin M(\tau + \epsilon)$
then we can not have $\gamma([0, \ell]) \subset M \setminus \Sigma$, since this implies that 
$l(\gamma) = \tilde l(\gamma) \le \tau(y) + \epsilon$,
whence a contradiction $x \in  M(\tau + \epsilon)$.
Thus $\gamma$ intersects $\Sigma$.
Let $t_1 > 0$ be the smallest $t \in [0,\ell]$
such that $\gamma(t) \in \p \Sigma$.
Moreover, let $t_0 \ge 0$ be the 
largest $t \in [0, t_1]$ such that $\gamma(t) 
\in M(\tau)$. 

We claim that $l(\gamma|_{[0,t_0]}) \ge \tau(y)$.
First, if $t_0 < t_1$ and $l(\gamma|_{[0,t_0]}) < \tau(y)$,
then for a small $t > 0$ we have that
$t_0 + t < t_1$ and $l(\gamma|_{[0,t_0 + t]}) < \tau(y)$.
Thus 
\begin{equation*}
t_0 + t < t_1 
\quad \text{and} \quad
\gamma(t_0 + t) \in M(\tau)^{int},
\end{equation*}
which is a contradiction with the maximality of $t_0$.
Second, if $t_0 = t_1$ and $l(\gamma|_{[0,t_0]}) < \tau(y)$,
then $\gamma(t_1) \in M(\tau)^{int} \cap \Sigma$,
which yields a contradiction with (\ref{boundary_intersection}).
Thus $l(\gamma|_{[0,t_0]}) \ge \tau(y)$.

We define $z := \gamma(t_1)$.
Then 
$l(\gamma|_{[t_0, t_1]}) \ge d(M(\tau),z)$.
Hence
\begin{align*}
\tilde d(x, z) + d(z, M(\tau))
&\le 
\tilde l(\gamma|_{[t_1, \ell]}) 
+ l(\gamma|_{[t_0, t_1]})
\\&= 
\tilde l(\gamma|_{[t_0, \ell]}) 
= \tilde l(\gamma) - l(\gamma|_{[0, t_0]})
\\&\le \tau(y) + \epsilon - \tau(y)
= \epsilon.
\end{align*}

%
%
\end{proof}

\begin{lemma}
\label{lem_the_difference_set}
Let $W \subset M$ be a neighborhood of 
$M(\tau) \cap \Sigma$.
Then there is $\epsilon_W > 0$ such that 
for all $\epsilon \in (0, \epsilon_W)$
\begin{align*}
\{ x \in \tilde M(\tau + \epsilon);\ 
\text{x satisfies (\ref{interesting_x})}\ \}
\subset 
W.
\end{align*}
\end{lemma}
\begin{proof}
We may choose a neighborhood $V$ of 
$M(\tau) \cap \Sigma$ such that 
$\bar V \subset W$.
Moreover, we may choose such $\epsilon_W > 0$ that
\begin{align}
\label{epsilon_W_1}
\epsilon_W &< \tilde d(\bar V, M \setminus W),
\\\label{epsilon_W_2}
\epsilon_W &< d(M(\Gamma, \tau), \Sigma \setminus V).
\end{align}
Let $\epsilon \in (0, \epsilon_W)$ and
let $x \in \tilde M(\tau + \epsilon)$
satisfy (\ref{interesting_x}).
By Lemma \ref{lem_interesting_x}
there is $z \in \Sigma$
such that $\tilde d(x, z) \le \epsilon_W$
and $d(z, M(\tau)) \le \epsilon_W$.
Then $z \in V$, since otherwise (\ref{epsilon_W_2})
would be violated.
Thus $x \in W$, since otherwise (\ref{epsilon_W_1})
would be violated.
\end{proof}

Let $\psi \in C^\infty(M)$ be a non-negative
function and define the measures
\begin{equation*}
m_\psi(E) := \int_{E} \psi dm,
\quad
\tilde m_\psi(E) := \int_{E} \psi d\tilde m.
\end{equation*}
The following theorem can be considered as a local version of Theorem \ref{thm_main}.
\begin{theorem}
\label{thm_main_local}
Let $x_0 \in M(\tau) \cap \Sigma$,
and suppose that there is a 
neighborhood $U(x_0)$ of $x_0$
such that 
$M(\tau) \cap U(x_0)$ is
an embedded smooth manifold with boundary.
Then there is a neighborhood $V(x_0) \subset M$
of $x_0$ and $\epsilon(x_0) > 0$ such that
for all $\epsilon \in (0, \epsilon(x_0))$,
\begin{align*}
\label{ineq_for_integrals_over_V_delta}
&m_\psi(V(x_0) \cap M(\tau + \epsilon))
 - \tilde m_\psi(V(x_0) \cap \tilde M(\tau + \epsilon))
\\&\quad\ge 
m_1(\epsilon, \psi, x_0)
- \O(\epsilon^2),
\end{align*}
where $m_1(\epsilon, \psi, x_0) \ge 0$.
Moreover, if $\psi(x_0) > 0$ then
\begin{equation*}
\liminf_{\epsilon \to 0} \epsilon^{-3/2} m_1(\epsilon, \psi, x_0) > 0.
\end{equation*}
\end{theorem}


Before proving Theorem \ref{thm_main_local},
let us introduce some notation and prove a couple of lemmas.
Let $x_0$ and $U(x_0)$ satisfy the 
assumptions of Theorem \ref{thm_main_local},
and let us consider such semi-geodesic coordinates 
in a neighborhood $U \subset U(x_0)$ of $x_0$ that $x_0 = 0$,
$\Sigma \cap U = \{ (x^1, x^2) \in B;\ x^1 \ge 0 \}$ and
\begin{equation}
\label{semi_geodesic_coordinates}
g(x^1, x^2) 
= (dx^1)^2 + h(x^1, x^2) (dx^2)^2,
\end{equation}
where $B$ is a neigborhood of the origin in $\R^2$
and $h$ is a stricly positive smooth function on $\bar B$.
We may choose $C_0 > 0$ such that 
for all $x = (x^1, x^2) \in U$ and $y = (y^1, y^2) \in U$,
\begin{equation}
\label{tilde_d_equivalent}
\frac{1}{C_0} \tilde d(x, y) 
\le |x^1 - y^1| + |x^2 - y^2| 
\le C_0 \tilde d(x, y).
\end{equation}

Let us define 
\begin{align*}
\rho(x) := d(x, M(\tau)),
\quad
\phi(x^2) := \rho(0, x^2).
\end{align*}
As $M(\tau) \cap U$ is a smooth manifold with boundary, the map $\rho$ is 
smooth in $U \setminus M(\tau)$ near $x_0$, whence also $\phi$ is 
smooth in a neighborhood of the origin. 
Non-negativity of $\phi$ and $\phi(0) = 0$
yield that $\phi'(0) = 0$.
In particular, there is a constant $a > 0$ such that 
\begin{equation}
\label{phi_domination}
\phi(x^2) \le a (x^2)^2
\end{equation}
in a neighborhood of origin.
Also $\rho(0) = 0$ and using boundary normal coordinates of $U \setminus M(\tau)$
we see that $\grad \rho(0) \ne 0$.
By (\ref{boundary_intersection}) we have that $\p_{x^1} \rho(0) > 0$.
Thus we may replace $U$ with a smaller neighborhood of $x_0$
still denoted by $U$ such that $\rho$ is smooth in $U$,
\begin{equation}
\label{distance_grows_with_x_1}
\p_{x^1} \rho(x) > 0, 
\quad x \in U \setminus M(\tau),
\end{equation}
and (\ref{phi_domination}) holds when $(0,x^2) \in U$.

\def\I{\mathcal I}
\begin{figure}[t]
\centering
\def\svgwidth{5cm}
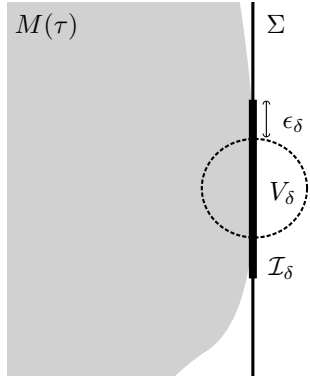
\caption{
A neighborhood of $x_0$ and the 
sets $\I_\delta$ and $V_\delta$.
}
\end{figure}
Let $\delta \in (0, 1)$.
As $\phi'(0) = 0$,
there is a neighborhood 
$\I_\delta \subset \R$ of origin such that 
$|\phi'(x^2)| \le \delta/C_0$ for all 
$x^2 \in \I_\delta$,
where $C_0$ is the constant in (\ref{tilde_d_equivalent}).
We may choose a neighborhood $V_\delta$ of $x_0$
and $\epsilon_\delta > 0$ such that 
for all $x = (x^1, x^2) \in M$
\begin{align*}
\tilde B(V_\delta, \epsilon_\delta)
\subset U, 
\quad
\tilde B(V_\delta, \epsilon_\delta) \cap \p \Sigma
\subset \{0\} \times \I_\delta, 
\end{align*}
where $\tilde B(V_\delta, \epsilon_\delta) := 
\{ x \in M;\ \tilde d(x, V_\delta) \le \epsilon_\delta\}$.
Next we study how the set 
$\tilde M(\tau + \epsilon) \cap V_\delta$ 
stretches in the $x^1$-direction compared to 
$M(\tau + \epsilon) \cap V_\delta$.
We show that the stretch is "of magnitude $c$"
in $\Sigma$ (Lemma \ref{lem_x_1_bound}), and that 
it is negligible in $M \setminus \Sigma$ 
(Lemma \ref{lem_x_1_bound_outside_Sigma}).

\begin{lemma}
\label{lem_phi_diff}
Let $\epsilon \in (0, \epsilon_\delta)$ and 
$x = (x^1, x^2) \in V_\delta \cap \tilde M(\tau + \epsilon)$.
If $x$ satisfies (\ref{interesting_x})
then there is $s \in \I_\delta$ such that
\begin{equation*}
\tilde d((x^1, x^2), (0, s)) \le \epsilon - \phi(s)
\le \frac{1}{1- \delta} (\epsilon - \phi(x^2))
+ \O(\epsilon^2).
\end{equation*}
In particular, if $\phi(x^2) \ge \epsilon$ 
then $|x^1| = \mathcal O(\epsilon^2)$.
\end{lemma}
\begin{proof}
By Lemma \ref{lem_interesting_x}
there is $z \in \p \Sigma$ such that
$\tilde d(x, z) \le \epsilon - d(z, M(\tau))$.
As $x \in V_\delta$ and $\epsilon < \epsilon_\delta$,
we have that $z \in U$.
Let us define $s \in \R$ by $(0, s) = z$.
Then $s \in \I_\delta$, and
\begin{equation*}
\tilde d(x, z) \le \epsilon - d(z, M(\tau))
= \epsilon - \phi(s).
\end{equation*}

By (\ref{tilde_d_equivalent}) we have that
$|x^2 - s| \le C_0 \tilde d(x, z) 
\le C_0 (\epsilon - \phi(s))
\le C_0 \epsilon$,
and by Taylor's theorem 
\begin{equation*}
\phi(x^2) = \phi(s) + \phi'(s) (x^2 - s) + R(x^2, s),
\end{equation*}
where $|R(x^2, s)| \le \norm{\phi''}_{L^\infty(\I_\delta)} C_0^2 \epsilon^2 = \O(\epsilon^2)$. Hence
\begin{align*}
\epsilon - \phi(s) 
&= \epsilon - \phi(x^2) 
+ \phi'(s) (x^2 - s) + R(x^2, s)
\\&\le 
\epsilon - \phi(x^2)
+ \frac{\delta}{C_0} C_0(\epsilon - \phi(s)) + R(x^2, s).
\end{align*}
Thus 
\begin{align*}
\epsilon - \phi(s)
&\le 
\frac{1}{1 - \delta} 
\ll(\epsilon - \phi(x^2) + R(x^2, s)\rr)
\\&= 
\frac{1}{1 - \delta} (\epsilon - \phi(x^2))
+ \O(\epsilon^2).
\end{align*}
By (\ref{tilde_d_equivalent})
we have that 
$|x^1| \le C_0 d((x^1, x^2), (0, s))$
whence $|x^1| = \O(\epsilon^2)$ if 
$\epsilon - \phi(x^2)$ is negative.
\end{proof}

\begin{lemma}
\label{lem_x_1_bound}
Let $\epsilon \in (0, \epsilon_\delta)$ and 
$x = (x^1, x^2) \in V_\delta \cap \tilde M(\tau + \epsilon)$.
If $x \in \Sigma$ then
\begin{align*}
x^1 
\le
\frac{c(0, x^2)}{1-\delta} 
(\epsilon - \phi(x^2)) + \O(\epsilon^2).
\end{align*}
\end{lemma}
\begin{proof}
If $x \in \p \Sigma$  then $x^1 = 0$. Thus we may assume
that $x \in \Sigma^{int}$.
Let $s \in \I_\delta$ be as in Lemma \ref{lem_phi_diff}.
Let 
$\gamma(t) = (\gamma^1(t), \gamma^2(t))$,
$t \in [0, \ell]$,
be a path from $(0, s)$ to $x$
in the manifold $(M, \tilde g)$
that is shorter than $\epsilon - \phi(s)$.
Let $a \in [0, \ell)$ be the parameter satisfying
$\gamma((a, \ell]) \subset \Sigma^{int}$
and $\gamma(a) \in \p \Sigma$.
We may assume that $\gamma|_{[a,\ell]}$
is a geodesic with respect to metric $\tilde g$.

Note that $\tilde d(\gamma(t), (0,x^2)) = \O(\epsilon)$.
Indeed, by (\ref{tilde_d_equivalent}) we have
\begin{align*}
\tilde d((0,x^2), (0,s))
&\le 
C_0 |x^2 - s|
\le C_0 (|x^1| + |x^2 - s|)
\\&\le 
C_0^2 \tilde d(x, (0,s))
\le C_0^2 \tilde l(\gamma) 
= \O(\epsilon),
\end{align*}
whence  
\begin{align*}
\tilde d(\gamma(t), (0,x^2))
&\le 
\tilde d(\gamma(t), (0,s))
+ \tilde d((0,s), (0,x^2))
\\&\le 
\tilde l(\gamma) + \O(\epsilon)
= \O(\epsilon).
\end{align*}
Moreover, 
\begin{align*}
c(\gamma(t)) 
&\le 
c(0, x^2) +
\norm{\grad c}_{L^\infty(\Sigma \cap U)} 
\tilde d(\gamma(t), (0,x^2))
\\&= c(0, x^2) + \O(\epsilon),
\quad t \in (a, \ell],
\end{align*}
whence by (\ref{semi_geodesic_coordinates})
\begin{align*}
x^1 
&= \gamma^1(\ell) 
= \int_a^\ell \p_s \gamma^1(s) ds
\\&\le \int_a^\ell 
  \frac{c(0,x^2) + \O(\epsilon)}{c(\gamma(s))} 
  \sqrt{|\p_s \gamma^1(s)|^2 + h(\gamma(s)) |\gamma^2(s)|^2}   
  ds
\\&= (c(0,x^2) + \O(\epsilon)) \int_a^\ell 
  |\p_s \gamma(s)|_{\tilde g} ds
\le (c(0,x^2) + \O(\epsilon)) \tilde l(\gamma)
\\&\le (c(0,x^2) + \O(\epsilon)) (\epsilon - \phi(s))
= c(0,x^2)(\epsilon - \phi(s)) + \O(\epsilon^2).
\end{align*}
By Lemma \ref{lem_phi_diff} we have
\begin{align*}
x^1 \le 
\frac{c(0,x^2)}{1-\delta} 
(\epsilon - \phi(x^2)) + \O(\epsilon^2).
\end{align*}
\end{proof}

\begin{lemma}
\label{lem_x_1_bound_outside_Sigma}
Let $\epsilon \in (0, \epsilon_\delta)$ and 
$x = (x^1, x^2) \in V_\delta \cap \tilde M(\tau + \epsilon)$.
Suppose that $x \notin M(\tau + \epsilon) \cup \Sigma$. Then $|x^1| = \O(\epsilon^2)$.
\end{lemma}
\begin{proof}
As $x \notin M(\tau + \epsilon)$,
we have that  
\begin{equation*}
0 < d(x, M(\tau + \epsilon)) \le \rho(x) - \epsilon.
\end{equation*}
Indeed, there is $z \in M(\Gamma, \tau)$,
and a path $\gamma : [0, \ell] \to M$
from $z$ to $x$ such that $l(\gamma) = \rho(x)$.
Moreover, there is $y \in \p M$ such that
$d(y, z) \le \tau(y)$.
We have $l(\gamma) > \epsilon$, since otherwise
$d(y, x) \le d(y, z) + l(\gamma) \le \tau(y) + \epsilon$,
whence a contradiction $x \in M(\tau + \epsilon)$.
Let $t \in [0, \ell]$ be the parameter satisfying
$l(\gamma|_{[0, t]}) = \epsilon$.
Then $d(y, \gamma(t)) \le \tau(y) + \epsilon$
and $\gamma(t) \in M(\tau + \epsilon)$.
Moreover, 
\begin{align*}
d(x, M(\tau + \epsilon)) 
\le d(x, \gamma(t)) \le l(\gamma|_{[t, \ell]})
= l(\gamma) - l(\gamma|_{[0, t]})
= \rho(x) - \epsilon.
\end{align*}

By (\ref{distance_grows_with_x_1})
the map, $s \mapsto \rho(s, x^2)$,
is increasing. Note that $x^1 < 0$ as $x \notin \Sigma$, 
whence 
\begin{equation*}
\epsilon < \rho(x^1, x^2) \le \rho(0, x^2) = \phi(x^2).
\end{equation*}
Thus $|x^1| = \O(\epsilon^2)$ by Lemma \ref{lem_phi_diff}.
\end{proof}

\begin{proof}[Proof of Theorem \ref{thm_main_local}]
Let $\epsilon \in (0, \epsilon_\delta)$.
We define
\begin{align*}
V(\epsilon) 
:= V_\delta \cap M(\tau + \epsilon) \cap \Sigma,
&\quad
\tilde V(\epsilon) 
:= V_\delta \cap \tilde M(\tau + \epsilon) \cap \Sigma,
\\
W(\epsilon)
:= V_\delta \cap M(\tau + \epsilon) \setminus \Sigma,
&\quad
\tilde W(\epsilon)
:= V_\delta \cap \tilde M(\tau + \epsilon) \setminus \Sigma.
\end{align*}
As $m_\psi(E) = \tilde m_\psi(E)$
for a measurable set $E \subset M \setminus \Sigma$,
Lemma \ref{lem_x_1_bound_outside_Sigma} yields
\begin{align*}
\tilde m_\psi(\tilde W(\epsilon)) 
- m_\psi(W(\epsilon))
= m_\psi(\tilde W(\epsilon) \setminus W(\epsilon))
= \O(\epsilon^2).
\end{align*}
Let us denote $\alpha := (1-\delta)^{-1}$.
By Lemma \ref{lem_x_1_bound}
\begin{equation}
\label{tilde_V_epsilon_domination}
\tilde V(\epsilon) 
\subset 
\{ (x^1, x^2);\ 
  0 \le x^1 \le \alpha c(0, x^2)(\epsilon - \phi(x^2)) 
  + \O(\epsilon^2),\
  x^2 \in \I_\delta \}.
\end{equation}
We define
$\I_\delta(\epsilon) = \{x^2 \in \I_\delta;\ \epsilon - \phi(x^2) \ge 0\}$ and
\begin{equation*}
\tilde V_1(\epsilon) 
:= 
\tilde V(\epsilon) \cap \{ (x^1, x^2) \in V_\delta;\
x^2 \in \I_\delta(\epsilon) \}.
\end{equation*}
Then $\tilde m_\psi(\tilde V(\epsilon) \setminus \tilde V_1(\epsilon)) = \O(\epsilon^2)$.

Note, that $\tilde \psi(x) := 
\psi(x) |g(x)|^{1/2}$
is smooth, whence
\begin{align*}
\tilde \psi(x)
\le 
\tilde \psi(0, x^2) + 
\norm{\p_{x^1} \tilde \psi}_{L^\infty(\Sigma \cap \bar U)} C \epsilon,
\quad x \in \tilde V(\epsilon).
\end{align*}
Similarly 
$c(x)^{-2} \tilde \psi(x) \le c^{-2}(0, x^2) \tilde \psi(0, x^2)
+ \O(\epsilon)$.
We denote 
$\underline c = \min_{\p \Sigma} c(x)$
and
\begin{equation*}
m_1(\epsilon, \psi) := 
\int_{\I_\delta(\epsilon)}
\tilde \psi(0, x^2) (\epsilon - \phi(x^2)) dx^2.
\end{equation*}
Then
\begin{align*}
&\tilde m_\psi(\tilde V_1(\epsilon))
=
\int_{\tilde V_1(\epsilon)} c(x)^{-2} \tilde \psi(x) dx
\\&\quad\le 
\int_{\I_\delta(\epsilon)} 
\int_0^{\alpha c(0, x^2)(\epsilon - \phi(x^2)) 
  + \O(\epsilon^2)} c(x)^{-2} \tilde \psi(x) dx^1 dx^2
\\&\quad\le
\int_{\I_\delta(\epsilon)} 
\ll( c(0, x^2)^{-2} \tilde \psi(0, x^2) + \O(\epsilon) \rr)
\alpha c(0, x^2) (\epsilon - \phi(x^2)) dx^2
+ \O(\epsilon^2)
\\&\quad\le
\alpha \underline c^{-1}\ m_1(\epsilon, \psi)
+ \O(\epsilon^2).
\end{align*}

By (\ref{semi_geodesic_coordinates}) we have
\begin{align*}
\{ (x^1, x^2);\ 
  0 \le x^1 \le \epsilon - \phi(x^2),\ 
  x^2 \in \I_\delta(\epsilon) \}
\subset V(\epsilon),
\end{align*}
whence
\begin{align*}
m_\psi(V(\epsilon))
&=
\int_{V(\epsilon)} \tilde \psi(x) dx
\ge 
\int_{\I_\delta(\epsilon)}
 \int_0^{\epsilon - \phi(x^2)} \tilde \psi(x) dx^1 dx^2
\\&=
m_1(\epsilon, \psi) - \O(\epsilon^2).
\end{align*}

By combining the above results we have
\begin{align*}
&m_\psi(V_\delta \cap M(\tau + \epsilon))
-\tilde m_\psi(V_\delta \cap \tilde M(\tau + \epsilon)) 
\\&\quad= 
m_\psi(V(\epsilon))
- \tilde m_\psi(\tilde V(\epsilon)) 
+ m_\psi(W(\epsilon))
- \tilde m_\psi(\tilde W(\epsilon)) 
\\&\quad= 
m_\psi(V(\epsilon)) 
- \tilde m_\psi(\tilde V_1(\epsilon))
- \O(\epsilon^2)
\\&\quad\ge
 (1 - \alpha \underline c^{-1}) m_1(\epsilon, \psi)
- \O(\epsilon^2).
\end{align*}
As $\underline c > 1$, we may choose 
small enough $\delta \in (0, 1)$
so that $1 - \alpha \underline c^{-1} > 0$.
We define
\begin{align*}
V(x_0) := V_\delta, 
\quad
\epsilon(x_0) := \epsilon_\delta,
\quad
\quad m_1(\epsilon, \psi, x_0) :=
(1 - \alpha \underline c^{-1}) m_1(\epsilon, \psi).
\end{align*}

By (\ref{phi_domination}) we have
$\I(\epsilon) := [-\sqrt{\epsilon / a}, \sqrt{\epsilon / a}] \subset \I_\delta(\epsilon)$
for small enough $\epsilon > 0$.
Note that 
\begin{align*}
\tilde \psi(0, x^2)
\ge 
\tilde \psi(0) -
\frac{C \norm{\p_{x^2} \tilde \psi(0, \cdot)}_{L^\infty(\I_\delta)}} 
{\sqrt{a}} \sqrt{\epsilon},
\quad x^2 \in \I(\epsilon).
\end{align*}
Hence
\begin{align*}
m_1(\epsilon, \psi)  
&\ge 
\int_{-\sqrt{\epsilon / a}}^{\sqrt{\epsilon / a}}
\tilde \psi(0, x^2) (\epsilon - a(x^2)^2) dx^2
\\&\ge
(\tilde \psi(0) - \O(\sqrt{\epsilon}))
\int_{-\sqrt{\epsilon / a}}^{\sqrt{\epsilon / a}} (\epsilon - a(x^2)^2) dx^2
\\&= 
(\tilde \psi(0) - \O(\sqrt{\epsilon}))
\frac{4}{3\sqrt{a}} \epsilon^{3/2}
= 
\frac{4 \tilde \psi(0)}{3\sqrt{a}}\epsilon^{3/2}
- \O(\epsilon^2).
\end{align*}
Hence $\psi(0) > 0$ implies that
\begin{align*}
\liminf_{\epsilon \to 0} \epsilon^{-3/2} m_1(\epsilon, \psi) > 0.
\end{align*}
\end{proof}

\begin{proof}[Proof of Theorem \ref{thm_main}]
As $M(\tau) \cap \Sigma$ is compact,
there are 
\begin{align*}
x_1, \dots, x_N \in M(\tau) \cap \Sigma
\end{align*}
such that with the neighborhoods $V_j := V(x_j)$,
$j = 1, \dots, N$, given by Theorem \ref{thm_main_local},
we have 
\begin{align*}
M(\tau) \cap \Sigma \subset \bigcup_{j=1}^N V_j.
\end{align*}
Let us choose a neighborhood $W$ of $M(\tau) \cap \Sigma$ such that $\bar W \subset \bigcup_{j=1}^N V_j$.
Let $\epsilon > 0$ satisfy
\begin{align*}
\epsilon < \epsilon_W
\quad \text{and} \quad
\epsilon < \epsilon(x_j),\ j=1, \dots, N,
\end{align*}
where $\epsilon_W$ is the constant of Lemma 
\ref{lem_the_difference_set}.
Then by Lemma \ref{lem_the_difference_set}
\begin{align*}
\tilde M(\tau + \epsilon) \setminus W
= M(\tau + \epsilon) \setminus W
\subset M \setminus \Sigma.
\end{align*}
Indeed, we have always $M(\tau + \epsilon) \subset \tilde M(\tau + \epsilon)$ and
if $x \in \tilde M(\tau + \epsilon) \setminus W$
then by Lemma \ref{lem_the_difference_set}
we have $x \in M(\tau + \epsilon)$
and $x \in M \setminus \Sigma$.
Hence 
\begin{align*}
\tilde m(\tilde M(\tau + \epsilon) \setminus W) = m(M(\tau + \epsilon) \setminus W).
\end{align*}

Let $(\psi_j)_{j=1}^N$ be a partition of unity
of $W$ subordinate to $(V_j)_{j=1}^N$
satisfying $\psi_j(x_j) \ne 0$, $j=1, \dots, N$.
Then by Theorem \ref{thm_main_local}
there are $c_1 > 0$ and $\epsilon_1 > 0$ such that
for all $\epsilon \in (0, \epsilon_1)$
\begin{align*}
&
m(M(\tau + \epsilon) \cap W)
- \tilde m(\tilde M(\tau + \epsilon) \cap W)
\\&\quad=
\sum_{j=1}^N \ll( m_{\psi_j}(M(\tau + \epsilon) 
\cap V_j)) - \tilde m_{\psi_j}(\tilde M(\tau + \epsilon) \cap V_j) \rr)
\\&\quad\ge
\sum_{j=1}^N m_1(\epsilon, \psi_j, x_j) - \O(\epsilon^2)
\ge
c_1 \epsilon^{3/2} - \O(\epsilon^2).
\end{align*}
Thus for small $\epsilon > 0$
\begin{align*}
\frac{m(M(\tau + \epsilon)) -\tilde m(\tilde M(\tau + \epsilon))}
{\epsilon^{3/2}}
\ge c_1 - \O(\epsilon^{1/2}).
\end{align*}
\end{proof}

\begin{figure}[t]
\centering
\def\svgwidth{5cm}
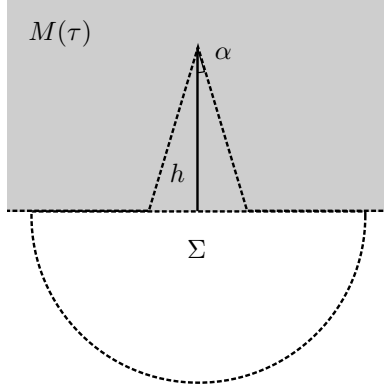
\caption{
The domain of influence $M(\tau)$
is the grey area, that is, the upper half plane.
The inclusion $\Sigma$ is the union of the 
lower half plane and the
triangle of height $h$.
Let $g$ be the Euclidean metric and
let $c(x)=2$, $x \in \Sigma$.
Then $\tilde M(\tau)$
contains the half disk of radius $h$.
}
\label{fig_sector_small_angle}
\end{figure}

\def\T{\mathcal T}
\def\H{\mathcal H}
\begin{example}
Let us consider the geometry described in Figure \ref{fig_sector_small_angle}.
We denote by $\T$ the triangle in the figure
and suppose that its height $h = 1$.
Moreover, let us denote by $\H$ the half disk in the figure. 
Then for small $\alpha > 0$ we have that
$\tilde m(\tilde M(\tau))
> m(M(\tau))$.
\end{example}

Indeed, we have for small $\alpha > 0$ that 
\begin{align*}
\tilde m(\tilde M(\tau) \cap \Sigma)
&= \frac{1}{4} m(\tilde M(\tau) \cap \Sigma)
\ge \frac{1}{4} \ll(m(\T) + m(\H) \rr)
\\&= 
\frac{1}{4} \ll(\tan \alpha + \pi/2 \rr)
> \tan \alpha = m(\T)
= m(M(\tau) \cap \Sigma).
\end{align*}
Moreover, it is always true that 
\begin{equation*}
\tilde m(\tilde M(\tau) \setminus \Sigma) 
= m(\tilde M(\tau) \setminus \Sigma) 
\ge m(M(\tau) \setminus \Sigma).
\end{equation*}

In the example, $\Sigma$ is non-smooth but it is 
clear that we may smoothen $\Sigma$ so that 
the change in the volume is negligible.
Thus the example shows that if we allow
$M(\tau)$ to penetrate deep into $\Sigma$
then the volume $\tilde m(\tilde M(\tau))$
may become larger than the volume $m(M(\tau))$.


\section{Distance to the inclusion}
\label{sec_distance}

To simplify the notation let us denote for $\tau \in C(\p M)$,
\def\vol{\mathcal V}
\begin{equation*}
\vol(\tau) := m(M(\tau)), 
\quad 
\tilde \vol(\tau) := \tilde m(\tilde M(\tau)).
\end{equation*}
In this section we prove the following two theorems.

\begin{theorem}
\label{thm_known_bg}
Suppose that $(\hat M, g)$ is a complete manifold without boundary.
Let $x \in \p M$, $r > 0$ 
and define $\tau_r(y) = r - \hat d(x,y)$.
Then
\begin{align*}
\hat d(x, \Sigma) = 
\max \{r > 0;\ \vol(\tau_r) = \tilde \vol(\tau_r) \}.
\end{align*}
\end{theorem}

\begin{theorem}
\label{thm_unknown_bg}
Suppose that $(M,g)$ is simple.
Let $x \in \p M$, $0 < h < r$ and define
$\tau_{r, h}(x) := r$ and $\tau_{r, h}(y) := h$
for $y \ne x$.
Then $d(x, \Sigma)$ is the supremum of the numbers
$r > 0$  with the following property:
\begin{itemize}
\item[(C)]
There is $h_0 \in (0, r)$ so that 
$\p_r^2\, \tilde \vol(\tau_{r, h})$
exists for almost all $h$ in $[0, h_0]$.
\end{itemize}
\end{theorem}

By an argument similar with the argument after Theorem \ref{thm_main}, we can choose such $\tau \in C(\p M)$ that
$\tilde M(\tau_{r, h}) = \tilde M(\tau)$.
Thus the volume data (\ref{volume_data}) can be used to determine if property (C) holds.
The claims (i) and (ii) in Theorem \ref{thm_summary} 
are consequences of 
Theorems \ref{thm_unknown_bg} and \ref{thm_known_bg}, respectively.

\begin{lemma}
\label{smoothness_of_intersection}
Let $(\hat M, g)$, $x \in \p M$ and $\tau_r$
be as in Theorem \ref{thm_known_bg}.
Then 
$\hat B(x, r) \cap M = M(\tau_r)$.
Moreover, if 
\begin{equation}
\label{tau_r_intersection}
M(\tau_r) \cap \Sigma \ne \emptyset
\quad \text{and}
\quad 
M(\tau_r) \cap \Sigma^{int} = \emptyset,
\end{equation}
then $M(\tau_r)$ is an embedded smooth manifold with 
boundary near $\Sigma$.
\end{lemma}
\begin{proof}
We have for all $y,z \in M$ that $\hat d(y, z) \le d(y,z)$
since $M \subset \hat M$.
Let $z \in M(\tau_r)$.
Then there is $y \in \p M$ such that
$d(y, z) \le \tau_r(y)$, whence
$\hat d(x,z) \le \hat d(x, y) + d(y, z) \le r$
and $z \in \hat B(x, r)$.

Conversely, let $z \in \hat B(x, r) \cap M$. 
If $z \in \p M$ then 
\begin{equation*}
d(z,z) = 0 \le r - \hat d(x,z) = \tau_r(z),
\end{equation*}
whence $z \in M(\tau_r)$.
Thus we may suppose that $z \in M^{int}$.
Let $\gamma : [0,\ell] \to \hat M$ be a shortest path
from $x$ to $z$.
Then there is $s \in [0, \ell)$ such that
$\gamma(s) \in \p M$ and $\gamma|_{(s,\ell]}$
is a geodesic of $(M^{int}, g)$.
Let us denote the length of $\gamma$ by $\hat l(\gamma)$.
Then 
\begin{align*}
d(\gamma(s), z) 
&\le 
l(\gamma_{[s,\ell]}) 
= 
\hat l(\gamma_{[s,\ell]}) 
= 
\hat l(\gamma) 
- \hat l(\gamma|_{[0,s]}) 
\\&\le 
r - \hat d(x, \gamma(s))
= 
\tau_r(\gamma(s)).
\end{align*}
Hence $z \in M(\tau_r)$.
We have shown that $\hat B(x, r) = M(\tau_r)$.

Let us now assume that (\ref{tau_r_intersection}) holds.
Then $\hat d(x, \Sigma) = r$
and a point $z_0$ in $\hat B(x, r) \cap \Sigma$
is a nearest point to $x$ on the boundary $\p \Sigma$.
Hence there is a neighborhood $U_{z_0} \subset M$ 
of $z_0$ such that for all $z \in U_{z_0}$
\begin{align*}
z \mapsto \hat d(x, z)\ \text{is smooth and}\ 
\nabla_z d(x, z) \ne 0,
\end{align*}
see e.g. the proof of \cite[Lem. 2.15]{Katchalov2001}.
As $\p \hat B(x, r)$ is a level set of $z \mapsto \hat d(x, z)$
we have that it is a smooth embedded manifold near $z_0$.
This is true for all $z_0 \in \hat B(x, r) \cap \Sigma$,
whence $\p \hat B(x, r)$ is a smooth embedded manifold near
$\Sigma$.
\end{proof}

\begin{proof}[Proof of Theorem \ref{thm_known_bg}]
If $r < \hat d(x, \Sigma)$ then 
Lemma \ref{smoothness_of_intersection}
yields that
\begin{align*}
M(\tau_r) = \hat B(x, r) \cap M 
\subset M \setminus \Sigma,
\end{align*}
whence also $\tilde M(\tau_r) = M(\tau_r)$
and $\tilde m = m$ on $\tilde M(\tau_r)$.

If $r = \hat d(x, \Sigma)$
then (\ref{tau_r_intersection}) is satisfied 
and Lemma \ref{smoothness_of_intersection}
and Theorem \ref{thm_main}
yield that for small enough $\epsilon > 0$,
$\tilde \vol(\tau_{r + \epsilon})
< 
\vol(\tau_{r + \epsilon})$.
\end{proof}

\begin{lemma}
\label{lem_smoothness_on_simple}
Let $\tau_{r, h}$ be as in Theorem \ref{thm_unknown_bg}.
Suppose that $(M,g)$ is simple and 
let $r_0 > 0$.
Then there is $h_0 > 0$ such that
$r \mapsto \vol(\tau_{r, h})$
is smooth near $r_0$
for almost all $h \in (0, h_0)$.
\end{lemma}
\begin{proof}
We denote
\begin{align*}
B(r) := \{z \in M;\ d(z, x) \le r\}, 
\quad
M(h) := \{z \in M;\ d(z, \p M) \le h\},
\end{align*}
where $x \in \p M$ is as in Theorem \ref{thm_unknown_bg}.
Moreover, we denote by $v(y) \in T_y(M)$
the interior unit normal vector of $\p M$ at $y$
and define
\begin{align*}
&\p_+ S_x(M) 
:= 
\{ \xi \in T_x(M);\ 
|\xi|_g = 1,\ (\xi, v(x))_g > 0 \},
\\
&\tau_M(\xi) 
:= 
\max \{t > 0;\ \exp_x(t \xi) \in M \},
\quad \xi \in \p_+ S_x(M).
\end{align*}
As $(M,g)$ is simple we have that $\tau_M$ is smooth and
the map
\begin{equation}
\label{normal_coords_at_p}
(t, \xi) \mapsto \exp_x(t\xi),
\quad 
0 < t \le \tau_M(\xi),\ 
\xi \in \p_+ S_x(M),
\end{equation}
gives coordinates in $M \setminus \{x\}$.
Moreover, for $y = (\tau_M(\xi), \xi)$ in $\p M \setminus \{x\}$
we have $(v(y), \p_t \exp_x(\tau_M(\xi) \xi))_g < 0$.
Thus 
\begin{align*}
\p_t d(\exp_x(t \xi), \p M)|_{t=\tau_M(\xi)}
&= 
(\text{grad}_y d(y, \p M), \p_t \exp_x(\tau_M(\xi) \xi))_g
\\&= 
(v(y), \p_t \exp_x(\tau_M(\xi) \xi))_g
\ne 0.
\end{align*}

We define 
$S 
:= 
\{ \xi \in \p_+ S_x(M);\ 
\tau_M(\xi) \ge r_0/2\}$.
By the implicit function theorem and compactness 
of $S$
there are $h_0 > 0$, a finite number of points 
$\xi_j \in S$, $j=1, \dots N$,
neighborhoods
\begin{align*}
\xi_j \in U_j \subset S,
\quad 
\tau(\xi_j) \in I_j \subset \R
\end{align*}
and smooth functions $\sigma_j$ such that
for all $\xi \in U_j$, $t \in I_j$ and $h \in [0, h_0)$
\begin{align*}
d(\exp_x(t\xi), \p M) = h
\quad \text{if and only if} \quad
t = \sigma_j(\xi, h),
\end{align*}
and that the sets $I_j \times U_j$ cover the points 
$(\tau_M(\xi), \xi)$, $\xi \in S$.
Furthermore, we may require that $h_0 < r_0/2$ and that 
the boundary normal coordinates are well-defined on  $M(h_0)$.

Let us define for $h \in (0, h_0)$ and 
$\xi \in \p_+ S_x(M)$,
\begin{align*}
F(h, \xi) := \exp_y(h v(y))|_
{y = (\tau_M(\xi), \xi)}.
\end{align*}
Then $F$ gives coordinates in 
$M(h_0) \setminus B(r_0/2)$ and 
the range of $dF$ at 
$z$ in $\p B(r_0) \cap M(h_0)^{int}$ 
is the whole tangent space $T_z(M)$.
Thus the transversality theorem, see e.g. \cite[Thm. 3.2.7]{Hirsch1994}, yields that
$F_h(\xi) := F(h, \xi)$ is transversal to $\p B(r_0) \cap M(h_0)^{int}$ for almost all $h \in (0, h_0)$.

Let $h \in (0, h_0)$ be such that 
$F_h$ is transversal to $\p B(r_0) \cap M(h_0)^{int}$.
To simplify the notation we denote
$\sigma_j(\xi) := \sigma_j(\xi, h)$
and $\sigma_j'(\xi) := \p_\xi \sigma_j(\xi)$.
Using the normal coordinates (\ref{normal_coords_at_p})
we may identify $T_{(t, \xi)}(M) = \linspan\{\p_t, \p_\xi\}$. 
The subspace $T_{(t, \xi)}(\p B(r_0))$ is then spanned by $\p_\xi$.
By transversality the vector $dF_h \p_\xi$ has nonzero $\p_t$ component 
at the points $(r_0, \xi) \in \p B(r_0) \cap M(h_0)^{int}$
satisfying $d((r_0, \xi), \p M) = h$.
Moreover, for $(t, \xi) \in U_j \times I_j$
the $\p_t$ component is 
\begin{align*}
dt(d F_h \p_\xi)
= 
\sigma_j'(\xi).
\end{align*}
Hence $\sigma_j'(\xi) \ne 0$
for all $\xi \in S$ satisfying 
$\sigma_j(\xi) = r_0$.

Let $(\phi_j(\xi))_{j = 1}^N$,
be a partition of unity in $S$
subordinate to the cover $(U_j)_{j=1}^N$.
The solutions of 
$\sigma_j(\xi) = r_0$, $\xi \in \supp(\phi_j)$,
can not have an accumulation point $\xi_\infty$
since otherwise we would have 
a contradiction $\sigma_j'(\xi_\infty) = 0$.
Thus the number of solutions is finite, and we may choose a finer partition of unity $(\psi_j)_{j=1}^M$
such that for any $j = 1, \dots, M$
there is at most one solution 
$\sigma_j(\xi) = r_0$, $\xi \in \supp(\psi_j)$.
Moreover, we choose $(\psi_j)_{j=1}^M$ so that $\sigma_j(\xi) = r_0$ implies 
$\xi \in \supp(\psi_j)^{int}$.
Here we have reindexed the functions $\sigma_j$, 
allowing $\sigma_j = \sigma_k$ for $j \ne k$,
so that $\sigma_j$ is defined in the support of $\psi_j$.

Let $r > r_0/2$. If
$(t, \xi) \in M \setminus B(r)$
then $\xi \in S$.
Moreover,
\begin{align*}
(t, \xi) \in M \setminus (B(r) \cup M(h))
\end{align*}
if and only if there is $j = 1, \dots, M$
such that $\xi \in \supp(\psi_j)$ and
$r < t < \sigma_j(\xi)$.
We define 
$S_j(r) 
:= 
\{\xi \in \supp(\psi_j);\ r < \sigma_j(\xi) \}$.
Then  
\begin{align}
\nonumber
&m(B(r) \cup M(h))
= m(M) - m(M \setminus (B(r) \cup M(h)))
\\&\quad=
\label{partition_of_complement_vol}
m(M) - 
\sum_{j=1}^M \int_{\mathcal S_j(r)} \psi(\xi) 
\int_r^{\sigma_j(\xi)} |g(s, \xi)|^{1/2} ds d\xi.
\end{align}
As $\sigma_j'(\xi)$ does not vanish in the 
solution set $\sigma_j(\xi) = r_0$
we get easily that $S_j(r_0) = \emptyset$
implies $S_j(r) = \emptyset$ 
for $r$ near $r_0$.
Moreover, $S_j(r_0) = \supp(\psi_j)$
implies $S_j(r) = \supp(\psi_j)$
for $r$ near $r_0$.
We see that the terms in the sum 
(\ref{partition_of_complement_vol})
corresponding these two cases are smooth.

Let us choose coordinates
$(-1,1) \ni \alpha \mapsto \xi \in \p_+ S_x(M)$
and let us consider such $j=1, \dots, M$
that there is a solution 
$\alpha_0 \in \supp(\psi_j)^{int}$ of
$\sigma_j(\alpha) = r_0$.
By the inverse function theorem 
$\sigma_j^{-1}$ is a smooth function near $r_0$.
Let us denote $\rho := \sigma_j^{-1}$ and 
define for small $\epsilon > 0$,
\begin{align*}
U := \{(r, \alpha) \in 
(r_0 - \epsilon, r_0 + \epsilon) \times (-1, 1);\ 
\rho(r) < \alpha \}.
\end{align*}
The set $U$ is open.
Let us assume for a moment that $\rho'(r_0) > 0$.
Then $(r, \alpha) \in U$
implies that $r < \sigma_j(\alpha)$ and the function 
\begin{align*}
(r, \alpha) \mapsto \psi_j(\alpha) 
\int_r^{\sigma_j(\alpha)} |g(s, \alpha)|^{1/2} ds
\end{align*}
is smooth on $U$.

Let us show for a smooth function $b$ that
\begin{align*}
V(r) := \int_{\mathcal S_j(r)} 
b(r, \alpha) d\alpha
\end{align*}
is smooth near $r_0$.
We have for small $t > 0$ and $r$ near $r_0$ that
\begin{align*}
\mathcal S_j(r) \setminus S_j(r + t) 
&= 
\{\alpha \in \supp(\psi_j);\ 
r < \sigma_j(\alpha) \le r + t \}
\\&= 
(\rho(r), \rho(r + t)],
\end{align*}
whence
\begin{align*}
&V(r) - V(r + t)
= 
\int_{S_j(r)} b(r, \alpha) d\alpha
-
\int_{S_j(r + t)} b(r + t, \alpha) d\alpha
\\&\quad= 
\int_{S_j(r)} b(r, \alpha) d\alpha
-
\int_{S_j(r + t)} b(r, \alpha) d\alpha
- t \int_{S_j(r + t)} \p_r b(r, \alpha) d\alpha
+ \O(t^2)
\\&\quad= 
b(r, \rho(r))(\rho(r + t) - \rho(r))
+ \int_{\rho(r)}^{\rho(r + t)} \O(t) d\alpha
\\&\quad\quad
- t \int_{S_j(r + t)} \p_r b(r, \alpha) d\alpha
+ \O(t^2).
\end{align*}
The case $t < 0$ is analogous, whence for 
$r$ near $r_0$,
\begin{align*}
V'(r) = 
- b(r, \rho(r)) \rho'(r) 
+ \int_{S_j(r)} \p_r b(r, \alpha) d\alpha.
\end{align*}
Clearly, the first term is a smooth function of $r$ near $r_0$. The second term is differentiable by the 
above argument. Hence $V$ is smooth near $r_0$ by induction.

The case $\rho'(r_0) < 0$ is analogous and we see that 
$r \mapsto m(B(r) \cup M(h))$ is smooth near $r_0$.
\end{proof}

\begin{proof}[Proof of Theorem \ref{thm_unknown_bg}]
Let $r_0 > 0$ satisfy $r_0 < d(x, \Sigma)$.
Then for $h < d(\Sigma, \p M)$ and $r$ near $r_0$,
we have $\vol(\tau_{r,h}) = \tilde \vol(\tau_{r,h})$.
Hence condition (C) holds by 
Lemma \ref{lem_smoothness_on_simple}.

Let $r_0 = d(x, \Sigma)$ and $0 < h < d(\Sigma, \p M)$.
Then (\ref{boundary_intersection}) is satisfied
for $\Gamma = \p M$ and 
$\tau = \tau_{r_0, h}$.
Moreover, using the coordinates 
(\ref{normal_coords_at_p})
we see that $M(\tau)$ is
an embedded smooth manifold with boundary
near $\Sigma$.
By Lemma \ref{lem_smoothness_on_simple}
there is $0 < h_0 < d(\Sigma, \p M)$ such that
$r \mapsto \vol(\tau_{r, h})$
is smooth near $r_0$ for almost all $h \in [0, h_0]$. 
Let us fix such $h \in [0, h_0]$ and 
denote 
$V(r) := \vol(\tau_{r, h})$
and $\tilde V(r) := \tilde \vol(\tau_{r, h})$.
Then 
\begin{align*}
&\frac
{\tilde V(r_0 + \epsilon) - 2 \tilde V(r_0) 
+ \tilde V(r_0 - \epsilon)}{\epsilon^{2}}
\\&\quad=
\frac
{\tilde V(r_0 + \epsilon) - V(r_0 + \epsilon)}
{\epsilon^{2}}
+ \frac
{V(r_0 + \epsilon) - 2 V(r_0) + V(r_0 - \epsilon)}
{\epsilon^{2}}. 
\end{align*}
The second term converges to $V''(r_0)$
and the first term diverges by Theorem \ref{thm_main}.
Hence $\tilde V''(r_0)$ does not exist, and
we see that (C) does not hold. 
\end{proof}

\section{The direct problem}
\label{sec_direct_problem}

In this section we study the regularity of the solution 
$u^f$ of the wave equation (\ref{eq:wave}).
In particular, we see by combining results  
\cite{Lasiecka1991} with results \cite{Lions1972},
that the Neumann-to-Dirichlet operator  
$\Lambda_{2T}$ is well defined on $L^2((0, 2T) \times \p M)$.

We define a quadratic form by
\begin{equation*}
Q_{\tilde g}(u, v) := \int_M (du, dv)_{\tilde g} d\tilde m,
\quad u, v \in H^1(M),
\end{equation*}
where $du$ denotes the exterior derivative of $u$ and
$(\cdot, \cdot)_{\tilde g}$ is the inner product 
on the contangent bundle given by $\tilde g$.
Then 
\begin{equation*}
\Delta_{\tilde g} : H^1(M) \to (H^1(M))^*
\end{equation*}
with homogeneous Neumann boundary conditions is defined by
\begin{equation*}
(\Delta_{\tilde g} u, v)_{L^2(M; \tilde m)} = -Q_{\tilde g}(u, v),
\quad u, v \in H^1(M).
\end{equation*}
By \cite[Thm 9.5]{Lions1972} the solution of the problem 
\begin{align}
\label{variational_wave_eq}
\p_t^2 u - \Delta_{\tilde g} u = F \in L^2(0, 2T; (H^{1-\theta}(M))^*),  
\quad u_{t < 0} = 0,
\end{align}
satisfies $u \in C([0, 2T]; H^\theta(M))$, 
where $\theta \in (0, 1)$ and $u_{t < 0} = 0$
stands for $u|_{t = 0} = \p_t u|_{t = 0} = 0$.
Let $\theta \in (0, 1/2)$,
$f \in L^2((0, T) \times \p M)$ and define 
\begin{align}
\label{def_F}
\pair{F(t), v} 
:= 
\int_{\p M} f(t, x) v|_{\p M}(x) dS_g,
\quad v \in H^{1-\theta}(M),
\end{align}
where $dS_g$ is the Riemannian volume measure of 
$(\p M, g|_{\p M})$.
Then $F$ is in the space $L^2(0, 2T; (H^{1-\theta}(M))^*)$
and the solution $u = u^f$ of (\ref{variational_wave_eq})
solves the equation (\ref{eq:wave}) in variational sense, see equation (9.21a) in \cite[p. 288]{Lions1972}.
In particular, the map 
$f \mapsto u^f(T)$ is a compact operator
\begin{align*}
L^2((0, T) \times \p M) \to L^2(M).
\end{align*}

The Laplace-Beltrami operator $\Delta_g$ of $(M, g)$
has smooth coefficients, and
by \cite{Lasiecka1991} the solution of the Neumann problem 
\begin{align*}
&\p_t^2 v - \Delta_g v = 0,
\quad v_{t < 0} = 0,
\quad \p_\nu v = f \in L^2((0, 2T) \times \p M)
\end{align*}
satisfies $v \in C([0, 2T]; H^\alpha(M))$
for any $\alpha < 3/5$.
Let us choose a cutoff function $\chi \in C^\infty(M)$
such that $\chi = 1$ in a neighborhood of $\p M$
and that $g = \tilde g$
in the support of $\chi$.
Let us define $w := u^f - \chi v$. Then $w$ 
satisfies (\ref{variational_wave_eq}) with 
\begin{align*}
F 
&= 
(\p_t^2 - \Delta_{\tilde g}) (-\chi v)
= 
-\chi (\p_t^2  - \Delta_g)v 
-[\Delta_g, \chi]v
= 
[\chi, \Delta_g] v.
\end{align*}
Note that the commutator $[\chi, \Delta_g]$
is a differential operator of order one. 
Moreover, as $\chi = 1$ in a neighborhood of $\p M$
we have that $[\chi, \Delta_g] = 0$
in the same neighborhood. 
Thus there are $\tilde \chi \in C_0^\infty(M)$
and $C > 0$ such that 
\def\E{\mathcal E}
\begin{align*}
&\pair{[\chi, \Delta_g]v, \phi}_{\E'(M) \times C^\infty(M)}
= 
\pair{[\chi, \Delta_g]v, \tilde \chi \phi}_{\D'(M) \times C_0^\infty(M)}
\\&\quad= 
\pair{v, [\chi, \Delta_g]^* \tilde \chi \phi}
\le
\norm{v}_{H^\alpha(M)} 
\norm{[\chi, \Delta_g]^* \tilde \chi \phi}_{H^{-\alpha}(M)}
\\&\quad\le C
\norm{v}_{H^\alpha(M)} 
\norm{\phi}_{H^{1-\alpha}(M)}.
\end{align*}
Hence 
$F = [\chi,\Delta_g] v \in C([0, 2T]; (H^{1-\alpha}(M))^*)$ 
and $w \in C([0, 2T]; H^\alpha(M))$.
Choosing $\alpha \in (1/2, 3/5)$ we see that
$\Lambda_{2T}$ is a compact operator
\begin{align*}
L^2((0, T) \times \p M) \to L^2((0, T) \times \p M).
\end{align*}

Note that, if $f \in C_0^\infty((0,2T) \times M)$
then $v$ and $F$ are smooth and by 
\cite[Thm 8.2]{Lions1972}
\begin{align}
\label{natural_smoothness}
&u^f \in C([0,2T]; H^1(M)), 
\quad
\p_t u^f \in C([0,2T]; L^2(M)), 
\\\nonumber
&\p_t^2 u^f,\ \Delta_{\tilde g} u^f \in L^2(0,2T; (H^1(M))^*).
\end{align}

\section{Computation of the volumes}
\label{sec_minimization}

In this section we show that the minimization 
algorithm of \cite{Oksanen2011}
works also with a piecewise smooth metric tensor.
We denote
\begin{align}
\label{def_R_J}
&R f(t) := f(2 T - t), 
\quad 
J f(t) := \frac{1}{2} \int_0^{2 T} 1_L(t,s) f(s) ds,
\\\nonumber
&K := J \Lambda_{2T} - R \Lambda_{2T} R J,
\quad
I f (t) := 1_{(0,T)}(t) \int_0^t f(s) ds,
\end{align}
where $L := \{ (t,s) \in \R^2; t + s \le 2 T,\ s > t > 0 \}$.
Moreover, we denote by 
$I^+$ the adjoint of $I$ in 
\begin{equation*}
\mathcal S := L^2((0, 2T) \times \p M; dt \otimes dS_g),
\end{equation*}
where $dS_g$ is the Riemannian volume measure of the manifold $(\p M, g|_{\p M})$.

Let us denote by $(\cdot, \cdot)$ and $\norm{\cdot}$ 
the inner product and the norm of $\mathcal S$.
For $\tau \in C(\p M)$
we define $\mathcal S(\tau)$
to be the set of $f \in \mathcal S$ satisfying
\begin{equation}
\label{eq:source_supp_condition}
\supp(f) \subset \{ (t, y) \in [0, T] \times \p M;\ t \in [T - \tau(y), T] \}.
\end{equation}
We claim that the regularized minimization problem,
\begin{equation}
\label{minimization_regularized}
\argmin_{f \in \mathcal S(\tau)} \ll ((f, K f) - 2(I f, 1) + \alpha \norm{f}^2 \rr), 
\quad \alpha > 0,
\end{equation}
has unique solution $f_\alpha$, $\alpha > 0$, 
that is also the solution of the linear equation, 
\begin{equation}
\label{the_linear_control_eq}
(P_\tau K P_\tau + \alpha) f = P_\tau I^+ 1,
\quad \alpha > 0,
\end{equation}
where $P_\tau$ is the orthogonal projection from $\mathcal S$ to $\mathcal S(\tau)$.
The proof is based on the following two reformulations
of the Blagovestchenskii identity,
\begin{align}
\label{inner_products}
(u^f(T), u^h(T))_{L^2(M; d\tilde m)} 
&= (f, K h),
\\\label{inner_product_with_1}
(u^f(T), 1)_{L^2(M; d\tilde m)} &= (I f, 1),
\end{align}
where $u^f$ and $u^h$ are the solutions of (\ref{eq:wave})
corresponding to the boundary sources $f, h \in L^2((0, 2T) \times \p M)$, respectively. For the original formulation of the identity see \cite{Blagovevsvcenskiui1966}.
The equations (\ref{inner_products})
and (\ref{inner_product_with_1}) together with 
compactness of $f \mapsto u^f(T)$ imply
the existence and uniqueness for the minimization problem 
(\ref{minimization_regularized}) 
by using exactly the same argument as in \cite{Oksanen2011}.

Let us verify next the equations (\ref{inner_products}) and 
(\ref{inner_product_with_1}).
By continuity of the map $f \mapsto u^f(T)$ and of the operators $\Lambda_{2T}, J, R$ and $I$
it is enough to show (\ref{inner_products})
and (\ref{inner_product_with_1})
for $f$ and $h$ in $C_0^\infty((0, 2T) \times \p M)$.
Then $u^f$ and $u^h$ have the regularity properties (\ref{natural_smoothness}).
We have 
\begin{align*}
&(\p_t^2 - \p_s^2) (u^f(t), u^h(s))_{L^2(M; d\tilde m)}
\\&\quad=
\pair{\p_t^2 u^f(t), u^h(s)}_{(H^1(M))^* \times H^1(M)}
- \pair{u^f(t), \p_s^2 u^h(s)}_{H^1(M) \times (H^1(M))^*}
\\&\quad=
- \tilde q(u^f(t), u^h(s))
+ \int_{\p M} f(t, x) u^h|_{\p M}(s, x) dS_g
\\&\quad\quad\quad
+ \tilde q(u^h(s), u^f(t))
- \int_{\p M} h(s, x) u^f|_{\p M}(t, x) dS_g
\\&\quad=
\int_{\p M} 
\ll(f(t, x) \Lambda_{2T} h(s, x) - h(s, x) \Lambda_{2T} f(t, x) \rr) dS_g,
\end{align*}
whence we obtain (\ref{inner_products}) by solving a one dimensional wave equation 
with vanishing initial and boundary conditions.
Note that, the right hand side of this wave equation is in 
$L^2((0, 2T) \times (0, 2T))$, whence the solution is
in $C([0,2T]; H^1(0, 2T))$ and $w(T, T)$ is well 
defined. This follows also from (\ref{natural_smoothness}).
Moreover, 
\begin{align*}
&\p_t^2 (u^f(t), 1)_{L^2(M; d\tilde m)}
=
\pair{\p_t^2 u^f(t), 1}_{(H^1(M))^* \times H^1(M)}
\\&\quad=
- \tilde q(u^f(t), 1)
+ \int_{\p M} f(t, x) dS_g
=
\int_{\p M} f(t, x) dS_g,
\end{align*}
and we obtain (\ref{inner_product_with_1}) by solving
an ordinary differential equation with vanishing initial conditions. 


\begin{figure}[t]
\centering
\includegraphics[width=0.9\textwidth]{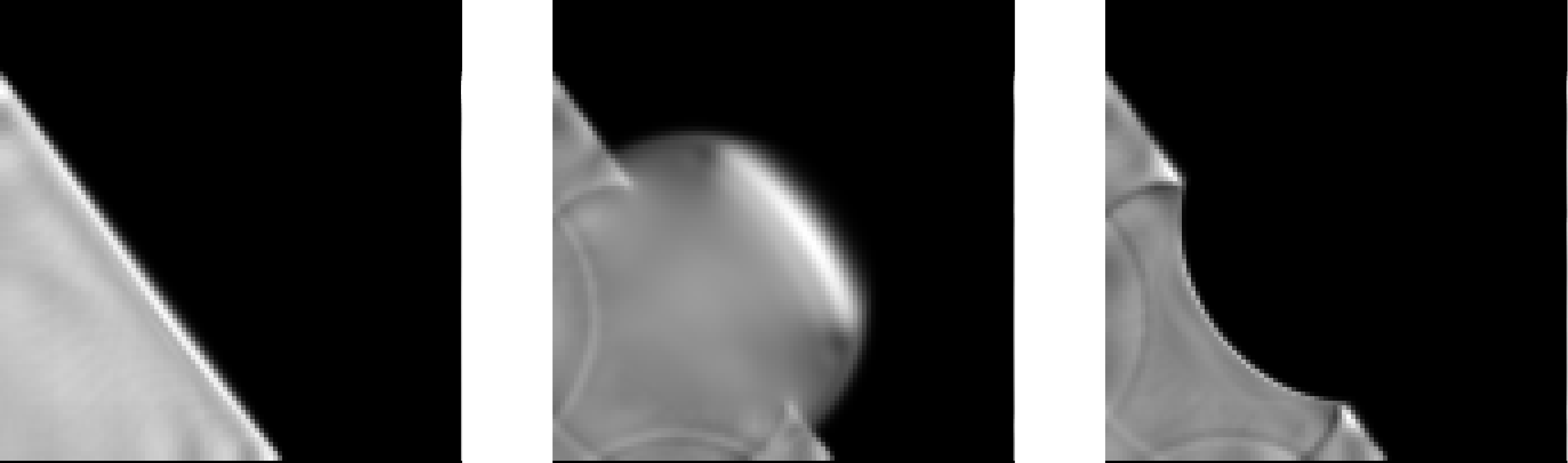}
\caption{
Waves $u^f(T)$ approximately
one on $\tilde M(\tau)$.
The boundary source $f$ has been 
obtained by solving a discretized version of (\ref{the_linear_control_eq})
with the conjugate gradient method.
Here $g$ is the Euclidean metric tensor and
we have chosen $\tau \in C(\p M)$ so that
$M(\tau)$ is a half plane.
{\em On left}, there is no inclusion. 
{\em In middle}, $\Sigma$ is a disk and $c(x) = 5$ for
$x \in \Sigma$. {\em On right}, 
$\Sigma$ is the same disk as in the middle
but $c(x) = 1/5$ for $x \in \Sigma$.
}
\label{fig_indicator_wave_fields}
\end{figure}

We claim that for all $\tau \in C_T(\p M)$,
\begin{equation}
\label{convergence_to_domi_indicator}
\lim_{\alpha \to 0} u^{f_\alpha}(T) =
1_{\tilde M(\tau)}
\quad 
\text{in $L^2(M)$},
\end{equation}
where $f_\alpha$, $\alpha > 0$, are the solutions of
(\ref{the_linear_control_eq}).
See Figure \ref{fig_indicator_wave_fields} for examples of 
computational such solutions of (\ref{the_linear_control_eq})
that $u^{f_\alpha}(T)$ approximates $1_{\tilde M(\tau)}$.
Note that $\Lambda_{2T}$ determines the volume data (\ref{volume_data}) by (\ref{convergence_to_domi_indicator})
and (\ref{inner_products}),
whence we have proved Theorem \ref{thm_summary} after proving 
(\ref{convergence_to_domi_indicator}).
Moreover, the convergence (\ref{convergence_to_domi_indicator}) 
can be proved exactly in the same way as the corresponding result in \cite{Oksanen2011}
given the following approximate controllability result.

\begin{lemma}
\label{lem_density}
Let $\Gamma \subset \p M$ be open, $T > 0$
and define $\tau(y) := T 1_\Gamma(y)$, $y \in \p M$. 
Then the embedding 
\begin{align*}
\{ u^f(T);\ f \in C_0^\infty((0,T) \times \Gamma) \}
\subset
L^2(\tilde M(\tau))
\end{align*}
is dense. 
\end{lemma}

This result is well-known
when the metric tensor is smooth, see e.g. \cite{Katchalov2001}.
We outline the proof in the appendix below.

\appendix
\section{Unique continuation}

In this appendix we employ the unique continuation result
by Tataru \cite{Tataru1995} to prove Lemma \ref{lem_density}.
We begin by proving two other lemmas.

\begin{lemma}
\label{lem_uniq_cont_through_discontinuity}
Let $T > 0$, $U \subset M^{int}$ be open
and define $\Gamma := U \cap \p \Sigma$.
Let 
\begin{align*}
u \in C([-T,T]; H^1(U)) \cap C^1([-T,T]; L^2(U))
\end{align*}
be a solution of 
$\p_t^2 u - \Delta_{\tilde g} u = 0$.
Then the traces $u|_{(-T, T) \times \Gamma}^+$
and $\p_\nu u|_{(-T, T) \times \Gamma}^+$
of $u|_{(-T, T) \times U \cap \Sigma}$
vanish if and only if the 
traces 
$u|_{(-T, T) \times \Gamma}^-$
and $\p_\nu u|_{(-T, T) \times \Gamma}^-$
of $u|_{(-T, T) \times U \setminus \Sigma}$
vanish.
\end{lemma}
\begin{proof}
Let us assume that the traces $u|_{(-T, T) \times \Gamma}^+$
and $\p_\nu u|_{(-T, T) \times \Gamma}^+$
exist and vanish on $(-T, T) \times \Gamma$.
We denote by $\tilde d_{\Sigma}$ the distance function of the Riemannian manifold $(\Sigma, \tilde g)$.
Let $\delta \in (0, T)$.
By Tataru's unique continuation, see \cite[Thm. 3.16]{Katchalov2001}
for a formulation suitable for our needs,
we have that $u(t)$ vanish on 
\begin{align*}
V_\delta := \{x \in U \cap \Sigma;\ 
\tilde d_{\Sigma}(x, \Gamma) < \delta \}
\end{align*}
for all $t \in (-T+\delta, T - \delta)$.
Thus $u$ satisfies the wave equation 
\begin{align}
\label{uniq_cont_we_switch_metric}
\p_t^2 u -\Delta_g u = 0 
\quad 
\text{on $(-T + \delta, T - \delta) \times V_\delta \cup (U \setminus \Sigma)$},
\end{align}
%
Again by unique continuation 
there is a neighborhood $U_\delta$ of
$\Gamma$ in $M$ such that $u$ vanish in 
$(-T + 2\delta, T - 2\delta) \times U_\delta$
for small $\delta > 0$.
Hence the traces 
$u|_{(-T, T) \times \Gamma}^-$
and $\p_\nu u|_{(-T, T) \times \Gamma}^-$
vanish on $(-T, T) \times \Gamma$,
and we have shown the necessity part in the claim of the lemma.
The sufficiency can be shown analogously 
after choosing a smooth Riemannian metric on $M$
that is an extension of $\tilde g$.
\end{proof}


\begin{lemma}
\label{lem_uniq_cont}
Let $T > 0$ and let 
\begin{align*}
u \in C([-T,T]; H^1(M)) \cap C^1([-T,T]; L^2(M))
\end{align*}
be a solution of 
$\p_t^2 u - \Delta_{\tilde g} u = 0$
satisfying for an open $\Gamma \subset \p M$
\begin{align*}
u|_{(-T, T) \times \Gamma} = 
\p_\nu u|_{(-T, T) \times \Gamma} = 0.
\end{align*}
Then $u(t,x)$ vanish when 
\begin{align*}
(t,x) \in (-T, T) \times M
\quad 
\text{and} 
\quad
\tilde d(x, \Gamma) + |t| < T.
\end{align*}
\end{lemma}
\begin{proof}
Let $(t_0,x_0) \in (-T, T) \times M$
satisfy $\tilde d(x_0, \Gamma) < T - |t_0|$.
Then there is $y \in \Gamma$
and a shortest path $\gamma : [0, \ell] \to M$
from $y$ to $x_0$ such that 
$\tilde l(\gamma) < T - |t_0|$.
The open set
$\gamma^{-1}(M \setminus \Sigma) \setminus \{0, \ell\}$
is a union of open disjoint intervals 
$(a_j, b_j)$, $j \in J$, where either $J=\N$ or $J = \{1, \dots, N\}$ 
for some $N \in \N$.

Let us show that $J = \N$ leads to a contradiction.
As the intervals are disjoint, we have
\begin{align*}
\sum_{j=1}^\infty \tilde l(\gamma|_{[a_j, b_j]})
\le \tilde l(\gamma).
\end{align*}
Thus 
$\tilde l(\gamma|_{[a_j, b_j]}) \to 0$ as $j \to \infty$.
Moreover, $\gamma(a_j), \gamma(b_j) \in \p \Sigma$,
and $\gamma|_{(a_j, b_j)}$ is a geodesic in 
$M \setminus \Sigma$.
Let us consider the semi-geodesic coordinates 
(\ref{semi_geodesic_coordinates}) in a neighborhood $U$ of 
$\gamma(a_j)$.
For large enough $j$ we have that 
$\gamma([a_j, b_j]) \subset U$.
Let us define $\gamma_j(t) := (0, \gamma^2(t))$,
$t \in [a_j, b_j]$.
Note that $\gamma_j$ is a smooth path from $\gamma(a_j)$
to $\gamma(b_j)$.

We will show that 
$\tilde l(\gamma|_{(a_j, b_j)}) 
> \tilde l(\gamma_j)$ for large $j$, which is a contradiction
since $\gamma$ is a shortest path. By (\ref{semi_geodesic_coordinates}) we have
\begin{align*}
\tilde l(\gamma|_{(a_j, b_j)}) 
&= 
\int_{a_j}^{b_j}
\sqrt{|\p_t \gamma^1(t)|^2 + h(\gamma(t)) |\p_t \gamma^2(t)|^2} dt
\\&\ge 
\int_{a_j}^{b_j}
\sqrt{h(\gamma(t))} |\p_t \gamma^2(t)| dt
\end{align*}
Moreover, for $t \in [a_j, b_j]$
\begin{align*}
\gamma^1(t) 
= \tilde d(\gamma(t), \p \Sigma) 
\le \tilde d(\gamma(t), \gamma(a_j))
\le \tilde l(\gamma|_{[a_j, t]})
\le \tilde l(\gamma|_{[a_j, b_j]}).
\end{align*}
Thus $\norm{\gamma^1}_{L^\infty(a_j, b_j)} \to 0$
as $j \to \infty$, whence for large $j$
\begin{align*}
\sqrt{h(0, \gamma^2(t))} <
\min_{z \in \p \Sigma} c(z) \sqrt{h(\gamma(t))},
\quad t \in [a_j, b_j].
\end{align*}
Moreover, $\int_{a_j}^{b_j} |\p_t \gamma^2(t)| dt > 0$
since otherwise $\gamma^2(t)$ is constant for $t \in [a_j, b_j]$ and $\gamma|_{[a_j, b_j]}$ is a loop.
Hence for large $j$ we have the contradiction
\begin{align*}
\tilde l(\gamma|_{(a_j, b_j)}) 
> 
\int_{a_j}^{b_j}
\frac{\sqrt{h(0, \gamma^2(t))}}{c(0, \gamma^2(t))} |\p_t \gamma^2(t)| dt
=
\tilde l(\gamma_j).
\end{align*}
We have shown that $J$ is finite.

We may renumber the intervals 
$(a_j, b_j)$, $j = 1, \dots, N$, so that
\begin{align*}
a_1 < b_1 \le a_2 < b_2 \le \dots \le a_N < b_N.
\end{align*}
Then we may repeatedly use unique continuation 
together with Lemma \ref{lem_uniq_cont_through_discontinuity}
either in $(M \setminus \Sigma, \tilde g)$
or in $(\Sigma, \tilde g)$ to see that $u(t)$ vanish near 
$\gamma(s)$ for $|t| < T - \tilde l(\gamma|_{[0, s]})$.
In particular, $u(t_0)$ vanish near $\gamma(\ell) = x_0$ 
since $|t_0| < T - \tilde l(\gamma)$.
\end{proof}

Given Lemma \ref{lem_uniq_cont}, the proof of Lemma \ref{lem_density}
is almost identical with the proof of \cite[Thm. 3.10]{Katchalov2001}.
Note that the boundary of $\tilde M(T 1_\Gamma)$
is of measure zero by the proof of \cite[Lem. 2]{Oksanen2011}.

\vspace{1cm}
{\em Acknowledgements.}
The author would like to thank Y. Kurylev for useful discussions. 
The research was partly supported by Finnish Centre of Excellence in Inverse Problems Research,
Academy of Finland COE 213476 and partly by Finnish Graduate School in Computational Sciences.

\bibliographystyle{abbrv}
\bibliography{main}

\begin{thebibliography}{10}

\bibitem{Anderson2004}
M.~Anderson, A.~Katsuda, Y.~Kurylev, M.~Lassas, and M.~Taylor.
\newblock Boundary regularity for the {R}icci equation, geometric convergence,
  and {G}el'fand's inverse boundary problem.
\newblock {\em Invent. Math.}, 158(2):261--321, 2004.

\bibitem{Astala2006}
K.~Astala and L.~P{\"a}iv{\"a}rinta.
\newblock Calder\'on's inverse conductivity problem in the plane.
\newblock {\em Ann. of Math. (2)}, 163(1):265--299, 2006.

\bibitem{Belishev1987}
M.~I. Belishev.
\newblock An approach to multidimensional inverse problems for the wave
  equation.
\newblock {\em Dokl. Akad. Nauk SSSR}, 297(3):524--527, 1987.

\bibitem{Belishev1999}
M.~I. Belishev and V.~Y. Gotlib.
\newblock Dynamical variant of the {BC}-method: theory and numerical testing.
\newblock {\em J. Inverse Ill-Posed Probl.}, 7(3):221--240, 1999.

\bibitem{Belishev1992}
M.~I. Belishev and Y.~V. Kurylev.
\newblock To the reconstruction of a {R}iemannian manifold via its spectral
  data ({BC}-method).
\newblock {\em Comm. Partial Differential Equations}, 17(5-6):767--804, 1992.

\bibitem{Bellassoued2010}
M.~Bellassoued and D.~D.~S. Ferreira.
\newblock Stability estimates for the anisotropic wave equation from the
  dirichlet-to-neumann map.
\newblock May 2010.

\bibitem{Bingham2008}
K.~Bingham, Y.~Kurylev, M.~Lassas, and S.~Siltanen.
\newblock Iterative time-reversal control for inverse problems.
\newblock {\em Inverse Probl. Imaging}, 2(1):63--81, 2008.

\bibitem{Blagovevsvcenskiui1966}
A.~S. Blagove{\v{s}}{\v{c}}enski{\u\i}.
\newblock The inverse problem of the theory of seismic wave propagation.
\newblock In {\em Problems of mathematical physics, {N}o. 1: {S}pectral theory
  and wave processes ({R}ussian)}, pages 68--81. (errata insert). Izdat.
  Leningrad. Univ., Leningrad, 1966.

\bibitem{Brown2003}
R.~M. Brown and R.~H. Torres.
\newblock Uniqueness in the inverse conductivity problem for conductivities
  with {$3/2$} derivatives in {$L^p,\ p>2n$}.
\newblock {\em J. Fourier Anal. Appl.}, 9(6):563--574, 2003.

\bibitem{Brown1997}
R.~M. Brown and G.~A. Uhlmann.
\newblock Uniqueness in the inverse conductivity problem for nonsmooth
  conductivities in two dimensions.
\newblock {\em Comm. Partial Differential Equations}, 22(5-6):1009--1027, 1997.

\bibitem{Bruhl2001}
M.~Br{\"u}hl.
\newblock Explicit characterization of inclusions in electrical impedance
  tomography.
\newblock {\em SIAM J. Math. Anal.}, 32(6):1327--1341 (electronic), 2001.

\bibitem{Burkard2009}
C.~Burkard and R.~Potthast.
\newblock A time-domain probe method for three-dimensional rough surface
  reconstructions.
\newblock {\em Inverse Probl. Imaging}, 3(2):259--274, 2009.

\bibitem{Chen2010}
Q.~Chen, H.~Haddar, A.~Lechleiter, and P.~Monk.
\newblock A sampling method for inverse scattering in the time domain.
\newblock {\em Inverse Problems}, 26(8):085001, 17, 2010.

\bibitem{Colton1996}
D.~Colton and A.~Kirsch.
\newblock A simple method for solving inverse scattering problems in the
  resonance region.
\newblock {\em Inverse Problems}, 12(4):383--393, 1996.

\bibitem{Dahl2009}
M.~F. Dahl, A.~Kirpichnikova, and M.~Lassas.
\newblock Focusing waves in unknown media by modified time reversal iteration.
\newblock {\em SIAM J. Control Optim.}, 48(2):839--858, 2009.

\bibitem{Gebauer2008}
B.~Gebauer.
\newblock Localized potentials in electrical impedance tomography.
\newblock {\em Inverse Probl. Imaging}, 2(2):251--269, 2008.

\bibitem{Greenleaf2003}
A.~Greenleaf, M.~Lassas, and G.~Uhlmann.
\newblock The {C}alder\'on problem for conormal potentials. {I}. {G}lobal
  uniqueness and reconstruction.
\newblock {\em Comm. Pure Appl. Math.}, 56(3):328--352, 2003.

\bibitem{Hahner1999}
P.~H{\"a}hner.
\newblock An inverse problem in electrostatics.
\newblock {\em Inverse Problems}, 15(4):961--975, 1999.

\bibitem{Hanke2008}
M.~Hanke, N.~Hyv{\"o}nen, and S.~Reusswig.
\newblock Convex source support and its applications to electric impedance
  tomography.
\newblock {\em SIAM J. Imaging Sci.}, 1(4):364--378, 2008.

\bibitem{Hirsch1994}
M.~W. Hirsch.
\newblock {\em Differential topology}, volume~33 of {\em Graduate Texts in
  Mathematics}.
\newblock Springer-Verlag, New York, 1994.
\newblock Corrected reprint of the 1976 original.

\bibitem{Ide2007}
T.~Ide, H.~Isozaki, S.~Nakata, S.~Siltanen, and G.~Uhlmann.
\newblock Probing for electrical inclusions with complex spherical waves.
\newblock {\em Comm. Pure Appl. Math.}, 60(10):1415--1442, 2007.

\bibitem{Ikehata1998}
M.~Ikehata.
\newblock Reconstruction of the shape of the inclusion by boundary
  measurements.
\newblock {\em Comm. Partial Differential Equations}, 23(7-8):1459--1474, 1998.

\bibitem{Ikehata2000}
M.~Ikehata.
\newblock Reconstruction of the support function for inclusion from boundary
  measurements.
\newblock {\em J. Inverse Ill-Posed Probl.}, 8(4):367--378, 2000.

\bibitem{Ikehata2004a}
M.~Ikehata.
\newblock Mittag-{L}effler's function and extracting from {C}auchy data.
\newblock In {\em Inverse problems and spectral theory}, volume 348 of {\em
  Contemp. Math.}, pages 41--52. Amer. Math. Soc., Providence, RI, 2004.

\bibitem{Isakov2009}
V.~Isakov.
\newblock Inverse obstacle problems.
\newblock {\em Inverse Problems}, 25(12):123002, 18, 2009.

\bibitem{Kabanikhin2005}
S.~I. Kabanikhin, A.~D. Satybaev, and M.~A. Shishlenin.
\newblock {\em Direct methods of solving multidimensional inverse hyperbolic
  problems}.
\newblock Inverse and Ill-posed Problems Series. VSP, Utrecht, 2005.

\bibitem{Katchalov2001}
A.~Katchalov, Y.~Kurylev, and M.~Lassas.
\newblock {\em Inverse boundary spectral problems}, volume 123 of {\em Chapman
  \& Hall/CRC Monographs and Surveys in Pure and Applied Mathematics}.
\newblock Chapman \& Hall/CRC, Boca Raton, FL, 2001.

\bibitem{Kirpichnikova2007}
A.~Kirpichnikova and Y.~Kurylev.
\newblock Inverse boundary spectral problem for riemannian polyhedra.
\newblock 2007.

\bibitem{Kirsch1998}
A.~Kirsch.
\newblock Characterization of the shape of a scattering obstacle using the
  spectral data of the far field operator.
\newblock {\em Inverse Problems}, 14(6):1489--1512, 1998.

\bibitem{Kohn1985}
R.~V. Kohn and M.~Vogelius.
\newblock Determining conductivity by boundary measurements. {II}. {I}nterior
  results.
\newblock {\em Comm. Pure Appl. Math.}, 38(5):643--667, 1985.

\bibitem{Kusiak2003}
S.~Kusiak and J.~Sylvester.
\newblock The scattering support.
\newblock {\em Comm. Pure Appl. Math.}, 56(11):1525--1548, 2003.

\bibitem{Lasiecka1991}
I.~Lasiecka and R.~Triggiani.
\newblock Regularity theory of hyperbolic equations with nonhomogeneous
  {N}eumann boundary conditions. {II}. {G}eneral boundary data.
\newblock {\em J. Differential Equations}, 94(1):112--164, 1991.

\bibitem{Lines2005}
C.~D. Lines and S.~N. Chandler-Wilde.
\newblock A time domain point source method for inverse scattering by rough
  surfaces.
\newblock {\em Computing}, 75(2-3):157--180, 2005.

\bibitem{Lions1972}
J.-L. Lions and E.~Magenes.
\newblock {\em Non-homogeneous boundary value problems and applications. {V}ol.
  {I}}.
\newblock Springer-Verlag, New York, 1972.
\newblock Translated from the French by P. Kenneth, Die Grundlehren der
  mathematischen Wissenschaften, Band 181.

\bibitem{Luke2003}
D.~R. Luke and R.~Potthast.
\newblock The no response test---a sampling method for inverse scattering
  problems.
\newblock {\em SIAM J. Appl. Math.}, 63(4):1292--1312 (electronic), 2003.

\bibitem{Luke2006}
D.~R. Luke and R.~Potthast.
\newblock The point source method for inverse scattering in the time domain.
\newblock {\em Math. Methods Appl. Sci.}, 29(13):1501--1521, 2006.

\bibitem{Oksanen2011}
L.~Oksanen.
\newblock Solving an inverse problem for the wave equation by using a
  minimization algorithm and time-reversed measurements.
\newblock {\em Inverse Probl. Imaging (to appear)}, Jan. 2011.

\bibitem{Paivarinta2003}
L.~P{\"a}iv{\"a}rinta, A.~Panchenko, and G.~Uhlmann.
\newblock Complex geometrical optics solutions for {L}ipschitz conductivities.
\newblock {\em Rev. Mat. Iberoamericana}, 19(1):57--72, 2003.

\bibitem{Pestov2010}
L.~Pestov, V.~Bolgova, and O.~Kazarina.
\newblock Numerical recovering of a density by the {BC}-method.
\newblock {\em Inverse Probl. Imaging}, 4(4):703--712, 2010.

\bibitem{Petkov2003}
V.~Petkov and L.~Stoyanov.
\newblock Sojourn times, singularities of the scattering kernel and inverse
  problems.
\newblock In {\em Inside out: inverse problems and applications}, volume~47 of
  {\em Math. Sci. Res. Inst. Publ.}, pages 297--332. Cambridge Univ. Press,
  Cambridge, 2003.

\bibitem{Potthast2001}
R.~Potthast.
\newblock {\em Point sources and multipoles in inverse scattering theory},
  volume 427 of {\em Chapman \& Hall/CRC Research Notes in Mathematics}.
\newblock Chapman \& Hall/CRC, Boca Raton, FL, 2001.

\bibitem{Potthast2006}
R.~Potthast.
\newblock A survey on sampling and probe methods for inverse problems.
\newblock {\em Inverse Problems}, 22(2):R1--R47, 2006.

\bibitem{Potthast2003}
R.~Potthast, J.~Sylvester, and S.~Kusiak.
\newblock A `range test' for determining scatterers with unknown physical
  properties.
\newblock {\em Inverse Problems}, 19(3):533--547, 2003.

\bibitem{Stefanov2005}
P.~Stefanov and G.~Uhlmann.
\newblock Stable determination of generic simple metrics from the hyperbolic
  {D}irichlet-to-{N}eumann map.
\newblock {\em Int. Math. Res. Not.}, (17):1047--1061, 2005.

\bibitem{Sylvester1987}
J.~Sylvester and G.~Uhlmann.
\newblock A global uniqueness theorem for an inverse boundary value problem.
\newblock {\em Ann. of Math. (2)}, 125(1):153--169, 1987.

\bibitem{Tataru}
D.~Tataru.
\newblock Carleman estimates, unique continuation and applications.
\newblock http://www.math.berkeley.edu/\~~tataru/papers/ucpnotes.ps.

\bibitem{Tataru1995}
D.~Tataru.
\newblock Unique continuation for solutions to {PDE}'s; between {H}\"ormander's
  theorem and {H}olmgren's theorem.
\newblock {\em Comm. Partial Differential Equations}, 20(5-6):855--884, 1995.

\bibitem{Uhlmann2009}
G.~Uhlmann.
\newblock Electrical impedance tomography and calderon's problem.
\newblock {\em Inverse Problems}, 25(12):123011, 2009.

\bibitem{Uhlmann2008}
G.~Uhlmann and J.-N. Wang.
\newblock Reconstructing discontinuities using complex geometrical optics
  solutions.
\newblock {\em SIAM J. Appl. Math.}, 68(4):1026--1044, 2008.

\end{thebibliography}
\end{document}